\documentclass[a4paper,12pt]{article}
    \usepackage[top=2cm,bottom=2cm,left=2.5cm,right=2.5cm]{geometry}
    \usepackage{amsmath, amssymb}
    \usepackage[margin=0.5cm,%
                font=small,%
                format=hang,%
                labelsep=period,%
                labelfont=bf]{caption}
\usepackage{graphicx}%
\usepackage{multirow}%
\usepackage{amsmath,amssymb,amsfonts}%
\usepackage{amsthm}%
\usepackage{mathrsfs}%
\usepackage[title]{appendix}%
\usepackage{xcolor}%
\usepackage{textcomp}%
\usepackage{manyfoot}%
\usepackage{booktabs}%
\usepackage{algorithm}%
\usepackage{algorithmicx}%
\usepackage{algpseudocode}%
\usepackage{listings}%
\usepackage{blkarray}
\usepackage{nicematrix}
\usepackage{enumitem}
\usepackage{cancel}
\usepackage{natbib}
\usepackage{fancyhdr}
\usepackage{lastpage}

\usepackage{amsmath, amssymb, booktabs, lastpage,}
\usepackage{multicol}
\usepackage{hyperref}
\usepackage{blkarray}
\usepackage{empheq} 
\usepackage{latexsym} 
\usepackage{array}
\usepackage{tikz-cd}
\usepackage{tikz}
\usetikzlibrary{arrows}
\usetikzlibrary{automata,positioning}
\usetikzlibrary{circuits.logic.US,circuits.logic.IEC,fit}

\numberwithin{table}{section}
\numberwithin{equation}{section}

\theoremstyle{plain}
\newtheorem{theorem}{Theorem}[section]

\newtheorem{definition}[theorem]{Definition}

\newtheorem{example}[theorem]{Example}

\newtheorem{remark}[theorem]{Remark}

\usepackage{tikz}
\usetikzlibrary{arrows, positioning, fit, calc}

\tikzset{
    species/.style={font=\Large\bfseries},
    inhibit/.style={red,-|,thick},
    activate/.style={red,->,thick},
    arrow/.style={-Latex,thick},
}

\renewcommand{\AA}{{\mathbf A}}
\newcommand{\NN}{{\mathcal N}}
\renewcommand{\SS}{{\mathcal S}}
\newcommand{\CC}{{\mathcal C}}
\newcommand{\RR}{{\mathcal R}}
\newcommand{\KK}{{\mathcal K}}
\newcommand{\dd}{{\rm d}}
\newcommand{\NetBiF}{{\sffamily{NetBiF}}}

\usepackage{authblk}
\author[1]{\textbf{Al Jay Lan J. Alamin}}
\author[1]{\textbf{Bryan S. Hernandez}}

\affil[1]{\small \textit{Institute of Mathematics, University of the Philippines Diliman, Quezon City 1101, Philippines}}
\affil[*]{Email addresses: \texttt{ajalamin@up.edu.ph} (corresponding author), \texttt{bshernandez@up.edu.ph}}

\title{Computational framework {\NetBiF} predicts and classifies steady-state bifurcations in biochemical systems via reaction network structure}
\date{\today}

\begin{document}
\maketitle
\begin{abstract}
    Dynamics of biochemical systems are dictated by the stoichiometry and kinetics of reaction networks. These systems may admit positive steady states that could undergo qualitative change in response to varying rate parameters. These changes are called steady-state bifurcations, and they correspond to biologically meaningful processes. Classical bifurcation analysis relies on exact kinetics and rate parameters, which are often unknown or difficult to measure for complex systems. To address this, structural bifurcation analysis (SBA) was introduced to predict bifurcations using only the system's associated reaction network. In this work, the computational package {\NetBiF} (\textbf{Net}work \textbf{Bi}furcation \textbf{F}inder) is introduced to implement SBA. {\NetBiF} is applied across diverse biochemical systems, including fundamental metabolic pathways, signaling and gene regulatory networks, and a large-scale carbon cycle model, to predict steady-state bifurcations. Crucially, for biochemical systems governed by power law kinetics, the implementation of {\NetBiF} also classifies the steady-state bifurcations as saddle-node, transcritical, or pitchfork via a structural analog of Sotomayor's theorem. To this end, explicit formulations for the bifurcation criteria were derived for these biochemical systems. Determining and classifying bifurcation behaviors provide insight into biochemical systems' abilities to adapt to changing and fluctuating conditions. Collectively, the study presents {\NetBiF} as a scalable, structure-based, and consistent computational tool to implement SBA. Our findings enable researchers to systematically study dynamical behaviors of biochemical systems through their reaction network structures without requiring explicit kinetic information.
    \\[1em]\noindent \textbf{Keywords:} chemical reaction network, biochemical systems, structural bifurcation analysis, flux-augmented matrix, steady states, biochemical pathways
\end{abstract}

\section{Introduction}\label{sec:intro}

Biochemical systems are governed by a large number of chemical reactions, which can collectively describe the molecular basis of the physiology of an organism. A central problem in mathematical and systems biology is to understand the long-term behavior of such systems. The methods of study for such systems include Chemical Reaction Network Theory (CRNT), which explores how the structure of a system's associated reaction network affects its dynamical behavior \citep{alon2019introduction,arceo2017reaction,feinberg2019crnt,voit2013BST}. In the CRNT framework, we model biochemical systems using \emph{reaction networks}, whose dynamics may then be determined by their stoichiometry and kinetics \citep{arceo2017reaction,feinberg2019crnt}. However, since most systems in biochemistry are either experimentally identified or incredibly large, the kinetics of reactions in these systems, as well as the values of the associated rate constants, are rarely known with full certainty. Hence, methods to analyze long-term behaviors of biochemical systems that minimize reliance on specific and exact kinetic parameters have also been heavily researched in recent years \citep{hernandez2023framework,hirono2025rethinking,JMP2019,Muller2012,Muller2014,okada2018sba,villejo2026}. 

One behavior being analyzed is the possibility of steady-state bifurcations in biochemical systems. Emerging methods hope to predict bifurcation solely based on the structure of the system's associated reaction network, eliminating the need to specify kinetic parameters. These methods are underlined by the theory of Structural Bifurcation Analysis (SBA) introduced by \cite{okada2018sba} and recently generalized by \cite{huang2025uncovering}. In the previous works, the method of SBA was established to identify the bifurcating species for each bifurcation parameter of biochemical systems, as well as determine the subnetworks which identify the conditions for which bifurcation can occur. Initially used to conduct sensitivity analysis \citep*{mochizuki2015sensitivity,okada2016localization,okada2017sensitivity}, the main driver for the prediction of bifurcation behavior via SBA is a square matrix constructed solely using the structure and topology of the associated reaction network, called the \emph{flux-augmented matrix} $\AA.$ The method posits an exact proportionality between the determinant of the Jacobian of the system and that of $\AA$ so that the search of zeros of the determinant of the Jacobian may instead be done on the flux-augmented matrix $\mathbf A,$ regardless of whether the system obeys any conservation laws \citep{okada2018sba,huang2025uncovering}.  

Despite these recent theoretical advances, SBA has yet to be implemented as a computational tool, and its applications remain limited to the hand-analyzed systems presented in previous papers. In this study, we address this gap and present a computational framework, which implements the method of SBA, to predict the possibility of steady-state bifurcations in different biochemical systems. In addition, for systems which are endowed with \emph{power law kinetics,} we classify the steady-state bifurcations as either saddle-node, transcritical, or pitchfork based on a structural classification method from \cite{okada2021flux}. We introduce {\NetBiF} (\textbf{Net}work \textbf{Bi}furcation \textbf{F}inder), a computational package developed in MATLAB to implement the prediction and classification of steady-state bifrucations via SBA, as well as the identification of \emph{buffering structures} and \emph{determinant structures}. 

The results of this work benefit the study of techniques in bifurcation analysis of complex biochemical systems. Indeed, by taking advantage of network structure and topology in studying bifurcation, we offer a scalable foundation for future work involving bifurcation analysis of large-scale biochemical models. Beyond this, we also contribute to the knowledge on the dynamics of the chosen biochemical systems as we discuss their bifurcation behaviors, and the type of bifurcations that their network structure admits. This study provides supporting evidence for the use of SBA in bifurcation analysis of biochemical systems, and promote further methodological development and application across many biochemical systems.

\section{Mathematics of Reaction Networks}\label{sec:MoRN}
We use Chemical Reaction Network Theory (CRNT) to define the \emph{reaction networks}, which model our biochemical systems \citep{feinberg1974dynamics,feinberg1987chemical,feinberg2019crnt,feinberg1979lectures,horn1972necessary,horn1972general}. Note that reaction networks do not necessarily permit a chemical interpretation, but it remains applicable to model our biochemical systems \citep{muller2022makes}.

As we focus on the implementation and application of SBA in this present paper, we provide the basic notions of CRNT in Appendix \ref{appendix:crnt}. To summarize, a reaction network $\NN$ and its associated kinetics $\KK$ give rise to a chemical \emph{kinetic system} $(\NN, \KK),$ which induces a system of ordinary differential equations (called \emph{species formation rate function}) describing the behavior of the concentrations of species in the system over time.

The study of kinetic systems is centered around the objective of understanding and predicting its long-term behavior. One approach is to derive the steady states of the system, and study their stability. Generally, this is quite difficult for biochemical systems due to their experimental nature, but when parameters are readily available and can be exactly determined, a framework to derive these analytic steady states has been developed \citep{hernandez2023framework}. In the same paper, they introduce a computational package to facilitate this derivation, called COMPILES (COMPuting anaLytic stEady States). Recently, a new computational framework called CRITERIA (Computing paRametrized posITive EquilibRIA) improves on COMPILES by considering a network's elementary flux modes \citep{villejo2026}. Importantly, the derivation of a closed-form parameterized steady state establishes monostationarity in a system. However, many biochemical processes rely on a system's ability to admit multiple steady states.  Multistationarity explains biological switch-like behaviors, which can dictate the fate of a cell or explain properties of gene regulation networks \citep{lai2004sonic,ozbudak2004multistability}. We wish to capture these switch-like dynamics, which mathematically arise from steady-state bifurcations. As such, the present paper introduces a new computational framework to analyze steady-state bifurcations of biochemical systems using only reaction network structure. 

\begin{example}
    \label{ex:main-example}
    Consider a reaction network consisting of three species, $X_1, X_2,$ and $X_3,$ and six reactions $R_1,\ldots, R_6$, namely \[\begin{array}{ccc}
    R_1\colon \ \emptyset\to X_1&
    R_2\colon \ X_1\to X_2&
    R_3\colon \ X_2\to X_1\\
    R_4\colon \ X_2\to X_3&
    R_5\colon \ X_3\to X_2&
    R_6\colon X_3\to \emptyset
\end{array}.\] The set of complexes in this reaction network is given by $\CC = \{\emptyset, X_1, X_2, X_3\}$ which is a subset of $\text{span}~\{X_1, X_2, X_3\}.$ Note that our reaction network induces a directed graph (digraph) $(\CC, \RR)$ where $\CC$ and $\RR = \{R_1, \ldots, R_6\},$ as defined above act as the nodes and arrows of the digraph, respectively. Hence, the network induces the digraph given by \begin{equation}
    \label{eq:network}
    \begin{tikzpicture}[baseline=(current  bounding  box.center), node distance = {15mm}]
            \node (0) {};
			\node (X1) [right of=0] {$X_1$};
			\node (X2) [right of=X1] {$X_2$};
			\node (X3) [right of=X2] {$X_3$};
            \node (1) [right of=X3] {};
            
            \draw [->] (0) -- (X1) node[midway,above]{$k_1$};
            
			\draw [-left to] ($(X1.east) + (0pt, 2pt)$) -- ($(X2.west) + (0pt, 2pt)$) node[midway,above]{$k_2$};
			\draw [-left to] ($(X2.west) + (0pt, -2pt)$) -- ($(X1.east) + (0pt, -2pt)$) node[midway,below]{$k_3$};

			\draw [-left to] ($(X2.east) + (0pt, 2pt)$) -- ($(X3.west) + (0pt, 2pt)$) node[midway,above]{$k_4$};
			\draw [-left to] ($(X3.west) + (0pt, -2pt)$) -- ($(X2.east) + (0pt, -2pt)$) node[midway,below]{$k_5$};

            \draw [->] (X3) -- (1) node[midway,above]{$k_6$};



            
	\end{tikzpicture}\end{equation}

    Following the definitions in the previous paragraphs, the stoichiometry of the system is dictated by its associated stoichiometric matrix given by\[
\nu =
\begin{bNiceArray}{cccccc}[
    first-row,
    last-col,
    margin
]
    R_1 & R_2 & R_3 & R_4 & R_5 & R_6 & {} \\
    1  & -1 &  1 &  0 &  0 &  0 & X_1 \\
    0  &  1 & -1 & -1 &  1 &  0 & X_2 \\
    0  &  0 &  0 &  1 & -1 & -1 & X_3 
\end{bNiceArray}.
\] Note that this matrix is full-ranked, that is, it is maximal, so that $\text{rank}~\nu = 3,$ which also equals the dimension of the network's stoichiometric subspace. Now, let $x_i$ be the concentration of species $X_i$ in the system for $i = 1,2, 3.$ If we endow the network with mass action kinetics, then the species formation rate function of the system is given by \[f(x;\mathbf k)=\frac{\dd}{\dd t} \begin{bmatrix}
        x_1\\x_2\\x_3
    \end{bmatrix}= \nu r(x;\mathbf k) =\begin{bmatrix}
        1&-1&1&0&0&0\\
        0&1&-1&-1&1&0\\0&0&0&1&-1&-1
    \end{bmatrix}\begin{bmatrix}
        k_1\\
        k_2x_1\\
        k_3x_2\\
        k_4x_2\\
        k_5x_3\\k_6x_3
    \end{bmatrix}.\] The steady states of the system are given by the points $(x_1^*, x_2^*, x_3^*)\in\mathbb R^3$ which satisfy
    \[f(x^*;\mathbf k)= \begin{bmatrix}
        k_1 - k_2x_1^* + k_3x_2^*\\
        k_2x_1^*-k_3x_2^* -k_4x_2^* + k_5x_3^*\\
        k_4x_2^*-k_5x_3^*-k_6x_3^*
    \end{bmatrix} = \begin{bmatrix}
        0\\0\\0
    \end{bmatrix}.\] Using COMPILES, the unique positive steady state of the system is given by \[x^* = (x_1^*, x_2^*, x_3^*)  = \left(\frac{k_1(k_3k_5 + k_3k_6 + k_4k_6)}{k_2k_4k_6}, \frac{k_1(k_5 + k_6)}{k_4k_6}, \frac{k_1}{k_6}\right).\] Having found a closed parameterized form, the (positive) steady state of the system is unique for each assignment of the rate parameters. 
\end{example}

\section{Classical vs. Structural Bifurcation Analysis}

Determining the possibility for significant and irreversible qualitative changes in long term behavior of biochemical systems is a worthwhile endeavor. Mathematically, this ability to undergo qualitative change is attributed to steady-state bifurcations. In this section, we compare two methods to establish steady-state bifurcations in dynamical systems: the Jacobian method and Structural Bifurcation Analysis (SBA). We restrict our discussion to local codimension-1 steady-state bifurcations where only one parameter near a steady state is varied.

\subsection{Classical Bifurcation Analysis}

Let $(\NN, \KK)$ be a kinetic system such that each reaction $R_i\in\RR$ in the system has an associated rate parameter $k_i >0.$ The vector of rate parameters in this system is given by $\mathbf k\in \mathbb R^N_{>0}$ where $N$ is the number of reactions in $\NN.$ Now, let $f(x;\mathbf k)$ be the species formation rate function of the system, and $x^*$ be a steady state. 

\begin{definition}
        The \emph{Jacobian of $f(x;\mathbf k)$} at $x = x^*$ is defined as $\nabla f({x^*};{\mathbf k}) = \nu\dfrac{\partial r(x; \mathbf k)}{\partial x}{\bigg |}_{x = x^*},$ where $\nu$ is the stoichiometric matrix of the network and $r(x;\mathbf k)$ is the vector of rate functions assigned to each reaction in the system describing its kinetics.  Moreover, let $\{\sigma_i\}_{i=1}^M$ be the eigenvalues of $\nabla f(x^*;\mathbf k)$ ordered by decreasing real parts ($\operatorname{Re} \sigma_1 \geq \operatorname{Re} \sigma_2\ge \ldots \geq \operatorname{Re}\sigma_M$).
        \end{definition}

If the real parts of all eigenvalues are negative, then we say that $x^*$ is \emph{(asymptotically) stable}. If at least one of the real parts of the eigenvalues takes a positive value, then $x^*$ is \emph{unstable}. Now, the steady state $x^*$ is said to undergo a \emph{bifurcation} when it changes stability as a rate parameter varies through a critical threshold. Since the eigenvalues $\{\sigma_i\}_{i=1}^M$ are ordered by decreasing real parts, a sufficient condition for the instability of $x^*$ as a steady state is if $\operatorname{Re} \sigma_1 > 0.$ It follows then that $x^*$ undergoes bifurcation whenever $\operatorname{Re} \sigma_1$ changes sign. For example, such a sign change when $\sigma_1$ is purely real indicates a steady-state bifurcation leading to a saddle-node, transcritical, pitchfork, or higher-order bifurcation \citep{strogatz2018nonlinear}. 

Since the determinant of the Jacobian is equal to the product of its eigenvalues, if \(\det\nabla f(x^*;\overline{\mathbf k}) =\sigma_1\sigma_2\cdots\sigma_M = \prod_{i=1}^M\sigma_i = 0\) then $x^*$ undergoes a bifurcation as $\mathbf k$ varies through $\mathbf{\overline k}.$ Therefore, a system undergoes a steady-state bifurcation as the bifurcation parameter varies through its threshold whenever the determinant of the Jacobian vanishes.

\begin{example}
    \label{ex:network-modified}
    Suppose the reaction network in \eqref{eq:network} is endowed with mass action kinetics and then modified such that $X_2$ inhibits its own production, and  $X_1$ activates the reaction $X_3\to X_2.$ Then, the rate functions for reactions $R_2$ and $R_5$ becomes \[r_2(x; k_2) = k_2x_1\frac{K_2^m}{K_2^m + (x_2)^m},\qquad r_5(x;k_3) = k_5x_3\frac{(x_1)^n}{K_3^n + (x_1)^n},\]
    where $m, n\in\mathbb N$ are the Hill coefficients, and $K_2, K_3$ are the associated dissociation constants for the regulatory reactions, respectively. Graphically, we can represent these new regulatory reactions by either a bar arrow ($\dashv$) or a regular arrow ($\to$) to denote an inhibition or activation regulation, respectively. Thus, we obtain the following reaction diagram for this modified system:
    \begin{equation}
    \label{eq:network-modified}\begin{tikzpicture}[baseline=(current  bounding  box.center), node distance = {15mm}]
            \node (0) {};
			\node (X1) [right of=0] {$X_1$};
			\node (X2) [right of=X1] {$X_2$};
			\node (X3) [right of=X2] {$X_3$};
            \node (1) [right of=X3] {};
            
            \draw [->] (0) -- (X1) node[midway,above]{$k_1$};
            
			\draw [-left to] ($(X1.east) + (0pt, 2pt)$) -- ($(X2.west) + (0pt, 2pt)$) node[midway,above]{$k_2$};
			\draw [-left to] ($(X2.west) + (0pt, -2pt)$) -- ($(X1.east) + (0pt, -2pt)$) node[midway,below]{$k_3$};

			\draw [-left to] ($(X2.east) + (0pt, 2pt)$) -- ($(X3.west) + (0pt, 2pt)$) node[midway,above]{$k_4$};
			\draw [-left to] ($(X3.west) + (0pt, -2pt)$) -- ($(X2.east) + (0pt, -2pt)$) node[midway,below]{$k_5$};

            \draw [->] (X3) -- (1) node[midway,above]{$k_6$};

            \draw[inhibit] (X2.north) 
    |- ($(X1)!0.5!(X2) + (0,30pt)$) 
    -- ($(X1)!0.5!(X2) + (0, 20pt)$)
    node[pos=0.25,left,red] {};

    \draw[activate] (X1.south) 
    |- ($(X2)!0.5!(X3) + (0,-30pt)$) 
    -- ($(X2)!0.5!(X3) + (0, -20pt)$)
    node[pos=1,left,red] {};

	\end{tikzpicture}.\end{equation}
    Solving for the determinant of the Jacobian in this case is quite difficult and tedious. Hence, we resort to numerical simulations to exhibit the system's bifurcation behavior, if any. Through numerical simulations, we find that an appropriate choice of parameters and rate functions for the regulatory reactions will result in a bifurcation phenomenon with $k_5$ being the bifurcation parameter. Figure \ref{fig:numerical_sim} shows the results of the numerical simulation where we see that a continuous change in the value of $k_5$ results in a qualitative change in the steady-state behavior of species $X_1$ and $X_2,$ but not in $X_3.$ 

    \begin{figure}
        \centering
        \includegraphics[width=\linewidth]{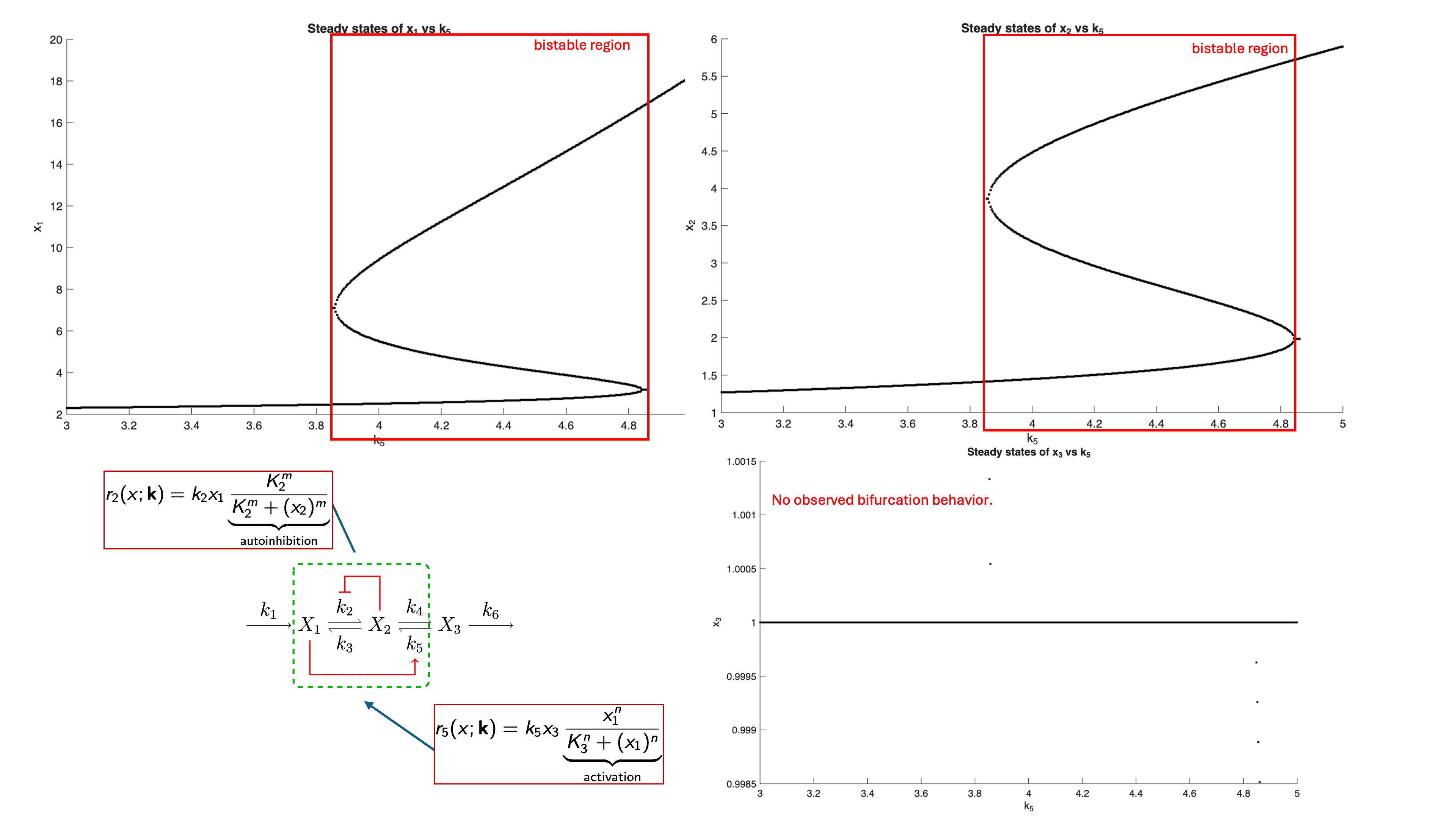}
        \caption{Numerical simulations for the steady-state concentrations of $X_1, X_2,$ and $X_3$ in \eqref{eq:network-modified}, with $k_5$ as bifurcation parameter. Saddle-node bifurcations were observed for $X_1$ and $X_2$ resulting in a bistable region. No bifurcation was observed for $X_3$ in the simulations. Sigmoidal repression and activation functions were chosen as such are sufficient to guarantee bistability in biochemical systems \citep{angeli2004detection}.}
        \label{fig:numerical_sim}
    \end{figure}
\end{example}


The main difficulty in the classical analysis of bifurcation is its requirement of exact and precise parameter values, which are difficult to obtain, if not unavailable, for most biochemical systems. Moreover, when conserved quantities are present in the system, the Jacobian can become singular and the classical method becomes ineffective in predicting steady-state bifurcations. These difficulties are circumvented by Structural Bifurcation Analysis (SBA). Parameter dependence in classical bifurcation analysis is replaced with structural and topological dependence only in SBA, and by the recent generalization of \cite{huang2025uncovering}, the method of SBA also becomes applicable for systems with conserved quantities  \citep{okada2016localization,okada2017sensitivity,okada2018sba,okada2021flux}. 

\subsection{Structural Bifurcation Analysis}
We begin our discussion of SBA with a definition of the \emph{flux-augmented matrix.}

\begin{definition}
\label{def:flux-aug}
    Let $(\NN, \KK)$ be a kinetic system with $M$ species and $N$ reactions, where $\nu$ is the $M\times N$ stoichiometric matrix of $\NN$ and $r =(r_1, r_2, \ldots, r_N)^\top$ is the vector of rate functions of $\KK.$ The \emph{flux-augmented matrix $\AA$} is a square matrix defined by \[\AA:=\begin{pmatrix}
        \dfrac{\partial r}{\partial x} & C\vspace{0.25cm}\\
        D^\top&0
    \end{pmatrix}\] where $C$  and $D$ are matrices whose columns form a basis for $\ker \nu$ and $\ker\nu^\top,$ respectively.
\end{definition}

Structurally, the flux-augmented matrix encodes the influences of reaction fluxes on the species and is given by the $\partial\mathbf r/\partial\mathbf x$ block. Moreover, the closed paths/loops of reactions in the network is given by a (right) kernel basis of the stoichiometric matrix denoted by $C.$ Finally, the $D^\top$ submatrix of $\AA$ encodes the existence of conserved quantities in the system.

Initially used to conduct sensitivity analysis, the main driver for the prediction of bifurcations via SBA is the flux-augmented matrix $\AA$ \citep{mochizuki2015sensitivity,okada2016localization,okada2017sensitivity}. It was shown in \cite{okada2018sba} that the determinant of the Jacobian induced by a system is proportional to the determinant of $\AA.$ Thus, the search for zeros of the determinant of the Jacobian may instead be done on the flux-augmented matrix $\mathbf A.$ In the original papers of SBA, the assumption is that the system does not admit any conserved quantities \citep{okada2018sba,okada2021flux}. Equivalently, this also means that there is no conservation law, which governs the system. Through \cite{huang2025uncovering}, the method of SBA is generalized to handle networks with conservation laws widening its applicability. The main theorem of \cite{huang2025uncovering} posits that even in the presence of conserved quantities, the determinant of $\AA$ is proportional to the Jacobian by some mapping so that identification behavior is still predicted by the zeros of the determinant of $\AA.$

\begin{example}
    \label{ex:example-sba}
    We have shown that the kinetic system in Example \ref{ex:main-example} is monostationary under mass action kinetics. Now, using the definition of the flux-augmented matrix on the reaction network representation of the system, we have \[
\AA = 
\begin{bNiceArray}{cccccc}[
    first-row,
    last-col,
    margin
]
    X_1 & X_2 & X_3 & C_1 & C_2 & C_3 & {} \\
    0 & 0 & 0 & 0 & 0 & 1 & R_1\\
    r_{2,1} & 0 & 0 & 1 & 0 & 1 & R_2\\
    0 & r_{3,2} & 0 & 1 & 0 & 0 & R_3\\
    0 & r_{4, 2} & 0 & 0 & 1 & 1 & R_4\\
    0 & 0 & r_{5,3} & 0 & 1 & 0 & R_5\\
    0 & 0 & r_{6,3} & 0 & 0 & 1 & R_6
\end{bNiceArray}
\] where the first three columns denote ${\partial r_i}/{\partial x_j} = r_{i,j}$ where $r_{i,j} \neq 0$ if species $X_j$ regulates reaction $R_i.$ In addition, we will have $r_{i,j} > 0$ if the regulation is positive, and $r_{i,j} < 0$ if negative. The last three columns represent a kernel basis of $\nu$ and encode the loops in the system. The system does not admit any conserved quantities so that $D = \{\mathbf 0\}.$ 

Solving for the determinant of $\AA$ gives $\det \AA = r_{2,1}r_{4,2}r_{6,3} \neq0$ since $r_{i,j} \neq 0,$ by definition. This means that $\det \AA\ne 0$ for any set of rate parameters, and thus will the Jacobian of the system. Therefore, any system with the associated reaction network above cannot admit any steady-state bifurcations. This supports our findings in Example \ref{ex:main-example}. Indeed, the uniqueness of the positive steady state for each assignment of the rate parameters rules out the possibility of any steady-state bifurcation behavior, as predicted here using SBA.
\end{example}

In the previous example, the determinant of $\AA$ was computed directly through cofactor expansion. We present a localization principle that arises from an appropriate reordering of rows and columns of $\AA$ forming a block matrix. This allows us to determine bifurcation behavior more efficiently as we take advantage of the independence of each substructure.

\subsection{Law of Localization and Buffering Structures}

Introduced by \cite{okada2016localization,okada2017sensitivity}, we state the \textit{law of localization} and the concept of \textit{buffering structures} to describe how perturbations on some parameters of a reaction network may sometimes be restricted to some parts of the network only.

\begin{definition}[Buffering Structures] \label{def:buffering-structure} A subnetwork $\Gamma = (\Gamma_\SS,\Gamma_\CC,\Gamma_\RR)$ of a reaction network is a \emph{buffering structure} if 
\begin{enumerate}
    \item for any $r\in \Gamma_\RR,$ we have $\partial r/ \partial x\neq 0$ for some $x\in \Gamma_\SS,$ and
    \item the \emph{index} given by \(\chi = M - N+\dim(\ker \nu) + \dim(\ker \nu^\top)\) is zero.
\end{enumerate}
\end{definition}

\begin{theorem}[Law of Localization, \cite{okada2016localization}] \label{thm:localization} If $\Gamma$ is a buffering structure, then any perturbation of reactions in $\Gamma$ does not affect the concentrations and fluxes outside of $\Gamma.$
\end{theorem}

Buffering structures provide a type of network decomposition, which allows us to factor $\det \AA.$ Indeed, from the proof of the Law of Localization from \cite{okada2016localization}, if $\Gamma$ is a buffering structure, then 
the rows and columns of $\AA$ may be permuted by permutation matrices $P_r$ and $P_c$ so that we have \[P_r\AA P_c=\begin{bmatrix}
    \AA_{\Gamma}&\ast\\
    0&\AA_{\Gamma'}
\end{bmatrix}\] where $\Gamma'$ is the complement of $\Gamma.$ The determinant of the above matrix is the product $\det\AA_\Gamma \cdot \det\AA_{\Gamma'}$ where the determinant of $\AA_\Gamma'$ is independent of the rate constants in $\Gamma.$ Then, the determinant of $\AA$ vanishes if and only if the determinant of the above matrix also vanishes. Specifically, we have $\det\AA = 0$ if and only if $\det \AA_\Gamma = 0$ or $\det\AA_{\Gamma'} = 0.$ As an example, the flux-augmented matrix $\AA$ in Example \ref{ex:example-sba} may be permuted so that we have
\[
P_r\AA P_c = 
\begin{bNiceArray}{c|cc|ccc}[
    first-row,
    last-col,
    margin
]
    X_1 & X_2 & C_1 & X_3 & C_2 & C_3 & {} \\
    r_{2,1} & 0 & 1 & 0 & 0 & 1 & R_2\\ \hline
    0 & r_{3,2} & 1 & 0 & 0 & 0 & R_3\\
    0 & r_{4,2} & 0 & 0 & 1 & 1 & R_4\\ \hline
    0 & 0 & 0 & r_{5,3} & 1 & 0 & R_5\\
    0 & 0 & 0 & r_{6,3} & 0 & 1 & R_6\\
    0 & 0 & 0 & 0 & 0 & 1 & R_1
\end{bNiceArray}.
\] This gives $\det(P_r\AA P_c) = r_{2,1}r_{4,2}r_{6,3} = \det \AA.$ We arrive at the same conclusion about the structural inability of the system to exhibit steady-state bifurcations, but the calculation of the determinant becomes much easier. The order of the rows and columns in the permuted matrix is not random as it is determined by the buffering structures that the network structure induces.

Now, if we assume that our kinetic system admits a stable steady state and its associated flux-augmented matrix $\AA$ is invertible, then the buffering structures of its reaction network are closed under unions and intersections \cite[Corollary 1]{hirono2021structural}. Consequently, the search for all buffering structures is reduced to finding such structures with minimal elements, since all larger buffering structures can then be constructed through unions or intersections. 

This allows us to efficiently isolate smaller subnetworks to check for bifurcation behavior. However, while these minimal buffering structures can localize these dynamics, the vanishing of $\det \AA$ often requires a finer decomposition. Although bifurcation occurs when the determinant of the flux-augmented submatrix corresponding to a buffering structure becomes zero, the fundamental structures which establish this are \emph{determinant structures,} defined as buffering structures with inner buffering structures removed \citep{huang2025uncovering}. Determinant structures need not be buffering structures, but buffering structures are always a union of some determinant structures. As a result, whenever a bifurcation condition is observed from a determinant structure, the method of SBA will conclude that the bifurcation is possible in the minimal buffering structure containing the determinant structure \citep[Appendix C]{okada2018sba}.

These results concerning the decomposition of a reaction network into its buffering structures allow us to study bifurcation behaviors in large biochemical systems. The law of localization and the concept of buffering structures give us a modular approach to structural bifurcation analysis by breaking down large and complex reaction networks into smaller independent ones. As such, we can analyze the independent contribution of each buffering structure to understand the system's overall bifurcation behavior.

\subsection{Classification of steady-state bifurcations in power law systems via reaction network structure} \label{sec:classify}

Beyond identifying the existence of bifurcating species and associated bifurcation parameters, it is biologically significant to determine the type of bifurcation a species exhibits at its bifurcation point. Since we consider only local steady-state bifurcations in this study, one can use the lemma by \cite{sotomayor1973structural}, as discussed in \cite[Ch.~4.2, Theorem 1]{perko2013differential}, to classify the type of bifurcation a system exhibits using its Jacobian. Having established the correspondence between the eigenvalues of the flux-augmented matrix and the system Jacobian, one can reformulate the sufficient conditions of Sotomayor using information from network structure and topology. Indeed, \cite{okada2021flux} presented analogous conditions to determine the type of bifurcation a system undergoes. 

Consider a dynamical system $f(x;\mathbf k)$ induced by a kinetic system, and suppose that it admits a steady state $x^*$ at the parameter $\mathbf k = \overline{\mathbf k}$. Let $\AA$ be the associated flux-augmented matrix evaluated in this steady state. Suppose $\AA$ has a simple zero eigenvalue with corresponding left and right eigenvectors $W$ and $U,$ respectively. We partition the right eigenvector as $U = [u^\top, v^\top]^\top,$ where $u$ represents the first $M$ components of $U$ and $v$ represents the remaining components. To facilitate the classification of steady-state bifurcations, define the following quantities involving the flux-augmented matrix and its derivatives:
\begin{description}[font=\normalfont]
    \item[(A0)] \(W_n\frac{\partial r_n}{\partial k_n}(x^*; \overline{\mathbf k}),\)
    \item[(A0*)] \(\frac{\partial r_n}{\partial k_n}(x^*; \overline{\mathbf k}),\)
    \item[(A1)] \(W_n\left[\nabla_x\left(\frac{\partial r_n}{\partial k_n}(x^*;\overline{\mathbf k})\right)\cdot u\right],\)
    \item [(H)] \(W^\top\nabla_x^2r(x^*;\overline{\mathbf k})(u,u):=W^\top\left[\sum_{m,m'}\frac{\partial^2 r(x^*; \overline{\mathbf k})}{\partial x_m\partial x_{m'}}u_mu_{m'}\right],\)
     \item [(T3)] \(W^\top\nabla_x^3r(x^*;\overline{\mathbf k})(u,u, u) := W^\top\left[\sum_{m, m', m''}\frac{\partial^3 r(x^*;\overline{\mathbf k})}{\partial x_m\partial x_{m'}\partial x_{m''}}u_mu_{m'}u_{m''}\right].\) 
\end{description}
\cite{okada2021flux} established that the capacity of the steady state $x^*$ to undergo specific types of steady-state bifurcations is determined by which of these defined quantities vanish. The proof of these results is a direct implication of the structural relation $f = \nu r$ and the decomposition $W = \nu^\top w$ of the left eigenvector of $\AA.$ The resulting classification criteria are summarized in Table \ref{tab:okada}.

\begin{table*}[h]
\centering
\caption{Summary of Okada et al.'s conditions for bifurcation classification 
via the flux-augmented matrix $\AA$. It is assumed that $\AA$ has a simple zero eigenvalue. Each of these conditions implies 
the corresponding Sotomayor condition \citep{sotomayor1973structural,okada2021flux}.}
\begin{tabular}{lllll}
\hline
\textbf{Cond.} & \textbf{Description} & \textbf{Saddle-Node} & \textbf{Transcritical} & \textbf{Pitchfork} \\
\hline
(A0)  & $W_n\frac{\partial r_n}{\partial k_n}$                   & $\neq 0$ & ---      & ---      \\
(A0*) & ${\partial r_n}/{\partial k_n}$                         & ---      & $= 0$    & $= 0$    \\
(A1)  & $W_n[\nabla_x(\frac{\partial r_n}{\partial k_n})\cdot u]$    & ---      & $\neq 0$ & $\neq 0$ \\
(H)  & $W^\top\nabla_x^2 r(u,u)$                                       & $\neq 0$ & $\neq 0$ & $= 0$    \\
(T3)  & $W^\top\nabla_x^3 r(u,u,u)$                                     & ---      & ---      & $\neq 0$ \\
\hline
\end{tabular}
\label{tab:okada}
\end{table*}

The above result of \cite{okada2021flux} is applicable to any system regardless of kinetics. However, implementing them computationally is challenging when rate functions remain arbitrary. Thus, we specify the classification criteria for the case of systems endowed with power law kinetics and derive computationally tractable formulations.

\begin{theorem}
    \label{thm:bif_power_law}
    Let $\Gamma$ be the minimal buffering structure containing the bifurcation condition at steady state $x^*$ and bifurcation parameter $\overline{\mathbf k}$, and let $\AA_{\Gamma}$ be the restriction of the flux-augmented matrix to $\Gamma.$ Assume $A_{\Gamma}$ has a simple zero eigenvalue, with left and right eigenvectors $W$ and $U = [u^\top, v^\top],$ respectively.

    If the system is endowed with power-law kinetics such that the rate function for each reaction $R_i\in\RR$ is $r_i(x) = k_i\prod_{j=1}^M x_j^{F_{ij}}$ with $F_{ij}$ being the kinetic order value of species $X_j$ in reaction $R_i,$ then the bifurcation classification quantities at steady state $x = x^*$ take the following forms:
    \begin{description}
    \item[(1) Net Sensitivity ($\Theta$):] 
    \[ \Theta = W^\top \frac{\partial r}{\partial k_n}{\bigg |}_{x = x^*}  = W_n\prod_{\ell = 1}^M (x_\ell^*)^{F_{n\ell}} \] Note that for positive steady states, $\Theta = 0$ if and only if $W_n =0 .$
    
    \item[(2)] Parameter Sensitivity Gradient for reaction $R_n$ (A1): 
    \[ W_n\left[\nabla_x\left(\frac{\partial r_n}{\partial k_n}\right)\cdot u\right] = W_n\frac{\partial r_n}{\partial k_n}{\bigg |}_{x=x^*}\cdot F_n^\top \tilde u,\qquad \tilde u = (u_i/x_i^*)_{i=1}^M\]
    
    \item[(3) Hessian Contraction (H):]
    \[ W^\top\nabla^2_x r(x^*;\mathbf{\overline k})(u,u) = \sum_{i} W_i r_i(x^*)\left((F_i^\top\cdot \tilde u)^2 - \sum_jF_{ij}\tilde u_j^2\right) \]
    
    \item[(4)] Third-Order Contraction (T3):
    \[ W^\top \nabla_x^3 r(x^*;\overline{\mathbf k})(u,u,u) = \sum_{i} W_i r_i(x^*)\left[(F_i^\top\cdot \tilde u)^3 - 3(F_i^\top\cdot \tilde u)\sum_jF_{ij}\tilde u_j^2 + 2\sum_jF_{ij}\tilde u_j^3\right] \]
\end{description}

\begin{proof}
    See Appendix \ref{appendix:proof-thm-2}.
\end{proof}
\end{theorem}

Theorem \ref{thm:bif_power_law} essentially streamlines the classification process for systems with power law kinetics, and allow us to have a computationally tractable formulation of the bifurcation classification conditions, as presented in Table \ref{tab:okada}.

\begin{remark} For the case of power law kinetics and in general, systems where the dependence of the rate function on the rate parameters is strictly monotonic, the net sensitivity parameter is nonzero for positive steady states. As such, we can exclude the possibility of transcritical and pitchfork bifurcations whenever the bifurcation occurs at a positive steady state. We restrict the occurrence of transcritical and pitchfork bifurcations to boundary steady states or at the origin. Thus, following Theorem \ref{thm:bif_power_law}, the value of (A1) is relevant only when $\Theta = 0$ to determine whether the system is structurally capable of exhibiting transcritical or pitchfork steady-state bifurcations. If the value corresponding to (A1) vanishes, then the theorem fails to classify the bifurcation. \end{remark}

\section{Implementation and Application of SBA to Biochemical Systems}
In this section, we introduce the computational package which implements the method of SBA and present its applications to different biochemical systems. The package, called {\NetBiF} (\textbf{Net}work \textbf{Bi}furcation \textbf{F}inder) (see https://github.com/ajlalamin/NetBiF), mainly uses the structure of the system's associated reaction network to determine the bifurcation candidates and associated bifurcation parameters of a kinetic system. The code is written in MATLAB and uses the input syntax of COMPILES \citep{hernandez2023framework}, a separate package which derives the analytic steady states of biochemical systems. At the highest level, the algorithm of {\NetBiF} takes a reaction network, along with its stoichiometric and kinetic matrices, generates a comprehensive bifurcation summary. This output gives the bifurcating species and parameters, as well as the minimal labelled buffering structure localizing the behavior. 

The algorithm for {\NetBiF} implements SBA through a procedural six-step pipeline. It uses both numeric and symbolic approaches to detect, filter, and output unique steady-state bifurcation candidates alongside the subnetwork responsible for such behavior. Using the user-defined reaction networks, whose syntax follows that of COMPILES \citep{hernandez2023framework}, stoichiometric and kinetic-order matrices are constructed. These matrices are used to construct the sign-preserving, randomized flux-augmented matrix $\mathbf A.$ Then, the package enumerates the buffering structures and applies a nested decomposition technique to determine the fundamental determinant structures. After the enumeration of these substructures, the numerical flux-augmented matrix is lifted into a symbolic representation where symbolic determinants are computed for each determinant structure. Finally, {\NetBiF} systematically filters and isolates conditions that can vanish to zero, thereby allowing for the detection of steady-state bifurcation candidates, as underlined by SBA. Figure \ref{fig:netbif-details} shows the details of this implementation of SBA via \NetBiF.

\begin{figure}
    \centering
    \includegraphics[width=\linewidth]{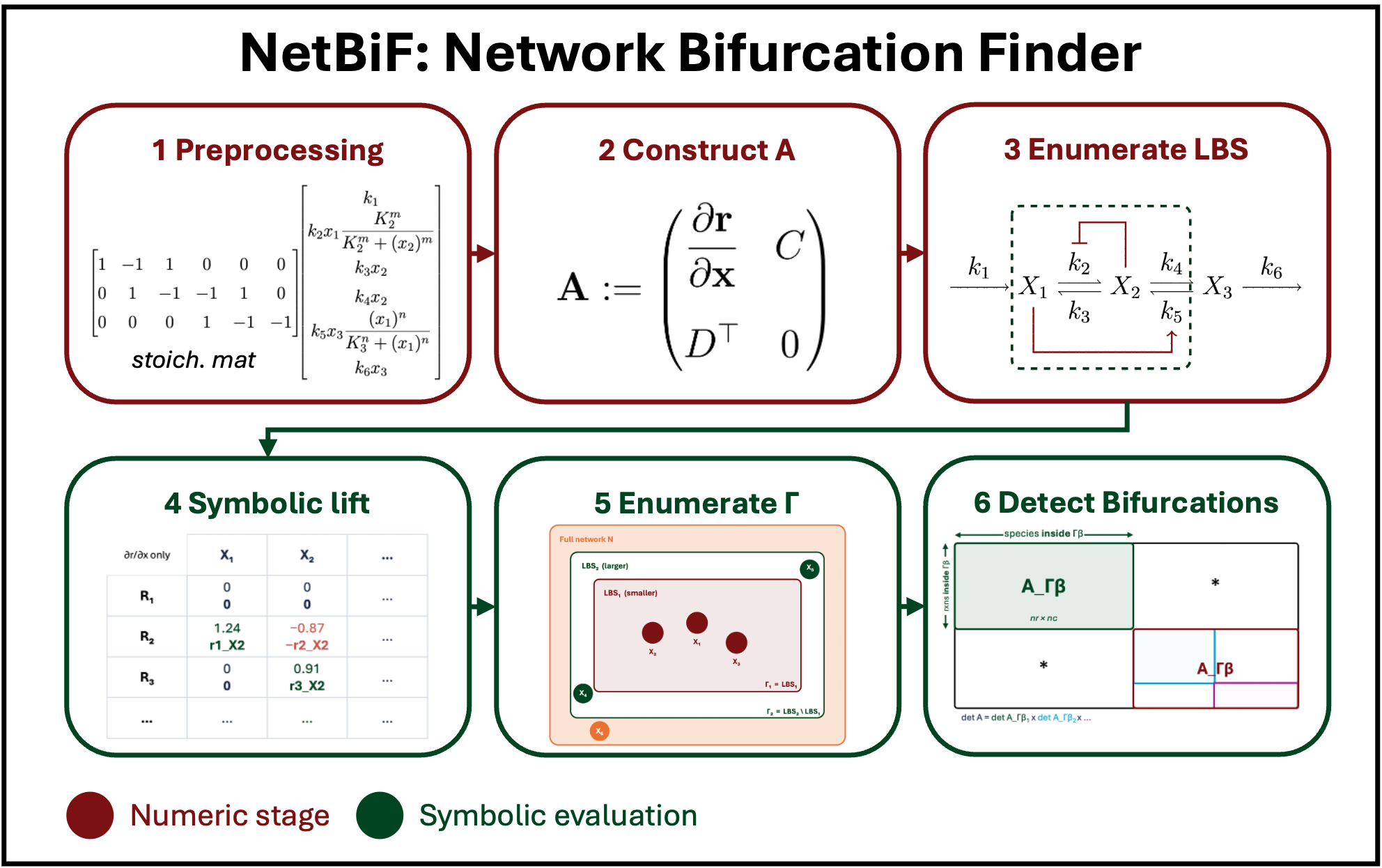}
    \caption[Details of the {\NetBiF} algorithm.]{{\textit{Details of the {\NetBiF} algorithm.} The six stages of the algorithm move between numeric (in maroon) and symbolic (in green) computations. (1) The network is parsed into its stoichiometric and kinetic-order matrices, (2) both of which are used to numerically construct the flux-augmented matrix $\AA$; (3) labelled buffering structures (LBS) are enumerated and sorted by size and (4) $\AA$ is lifted to a symbolic matrix while ensuring that the sign structure is preserved; finally, (5) LBS are peeled into determinant structures and (6) the determinant of corresponding submatrices are symbolically computed and factored to detect bifurcation candidates.}}
    \label{fig:netbif-details}
\end{figure}

Beyond determining the existence of steady-state bifurcations, \NetBiF classifies the predicted bifurcation as either saddle-node, transcritical, or pitchfork, for biochemical systems with power law kinetics. We use the classification procedure from \cite{okada2021flux}, and the derived quantities discussed in Section \ref{sec:classify} of this paper. Figure \ref{fig:classify_flowchart} shows the decision tree for the implementation of the structural classification of steady-state bifurcations.

\begin{figure}
    \centering
    \includegraphics[width=\linewidth]{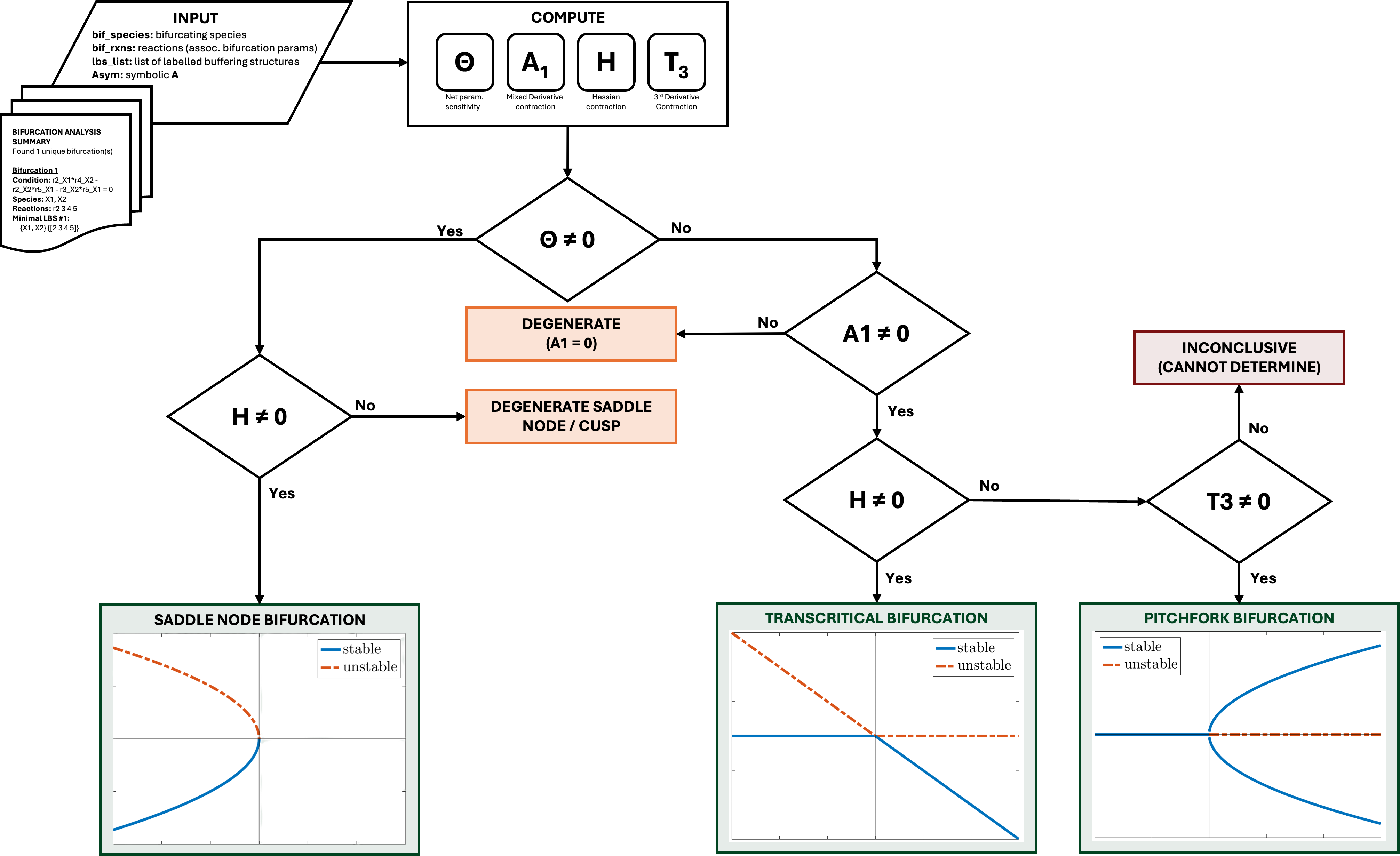}
    \caption[Decision tree for the structural classification of bifurcations.]{\textit{Decision tree for the structural classification of steady-state bifurcations in biochemical systems.} The method uses the bifurcating species and reactions, the list of labelled buffering structures, and the symbolic flux-augmented matrix to compute the net parameter sensitivity ($\Theta$), the mixed derivative contraction ($A_1$), the Hessian contraction ($H$), and the third derivative contraction ($T_3$) of the system. The successive structural non-vanishing of these quantities classify the steady-state bifurcation as saddle-node, transcritical, or pitchfork. Failure of relevant conditions implies a degenerate bifurcation (or cusp) or an inconclusive classification.}
    \label{fig:classify_flowchart}
\end{figure}

\begin{example}
    Recall the directed graph in \eqref{eq:network} induced by the reaction network of our running example. In Example \ref{ex:main-example}, we have shown that when the system is endowed with mass action kinetics, then it admits a unique positive steady state for each assignment of rate constants. This result was supported by Example \ref{ex:example-sba} where the determinant of the network's associated flux-augmented matrix $\AA$ was manually computed and shown to be nonvanishing for any positive rate constants. Now, we input the system's reaction network to {\NetBiF} to replicate these findings systematically. 
    
    First, {\NetBiF} constructs the flux-augmented matrix of the reaction network and identifies two labelled buffering structures: $$\Gamma_1 = (\{X_1\}, \{R_2, R_3\})\text{ and }\Gamma_2 = (\{X_1, X_2\}, \{R_2, R_3, R_4, R_5\}).$$ Then, the package computes for $\det\AA$ symbolically and finds that $\det \AA = r_{2,1} r_{4,2} r_{6,3}.$ Since $\det\AA\ne 0$ for the positive rate constants, the program outputs ``No bifurcation candidates found," concluding that the network's structure precludes steady-state bifurcations. We recreate the induced digraph of the reaction below along with the buffering structures identified by \NetBiF:
    $$\begin{tikzpicture}[baseline=(current  bounding  box.center), node distance = {15mm}]
            \node (0) {};
			\node (X1) [right of=0] {$X_1$};
			\node (X2) [right of=X1] {$X_2$};
			\node (X3) [right of=X2] {$X_3$};
            \node (1) [right of=X3] {};
            
            \draw [->] (0) -- (X1) node[midway,above]{$k_1$};
            
			\draw [-left to] ($(X1.east) + (0pt, 2pt)$) -- ($(X2.west) + (0pt, 2pt)$) node[midway,above]{$k_2$};
			\draw [-left to] ($(X2.west) + (0pt, -2pt)$) -- ($(X1.east) + (0pt, -2pt)$) node[midway,below]{$k_3$};

			\draw [-left to] ($(X2.east) + (0pt, 2pt)$) -- ($(X3.west) + (0pt, 2pt)$) node[midway,above]{$k_4$};
			\draw [-left to] ($(X3.west) + (0pt, -2pt)$) -- ($(X2.east) + (0pt, -2pt)$) node[midway,below]{$k_5$};

            \draw [->] (X3) -- (1) node[midway,above]{$k_6$};

\draw[green!70!black,dashed,very thick,rounded corners]
  ($(X1.north west)+(0.05, 0.5)$) rectangle ($(X3.south west)+(+0.05,-0.5)$);

\draw[orange!70!black,dashed,very thick,rounded corners]
  ($(X1.north west)+(0.07, 0.35)$) rectangle ($(X2.south west)+(+0.05,-0.35)$);
	\end{tikzpicture} \implies \text{no bifurcation candidates.}$$

Next, we modify the reaction network by adding two new regulatory reactions, an activation and an inhibition regulation, as we did in Example \ref{ex:network-modified}. We recall that a classical bifurcation analysis using the Jacobian could not be performed, and we resorted to numerical solutions to verify the existence of bifurcation behavior in the system (see Figure \ref{fig:numerical_sim}). From there, we found that as we varied the value of the parameter $k_5,$ the steady states of $X_1$ and $X_2$ exhibited steady-state bifurcation behaviors. In contrast, the steady state value of $X_3$ remained constant despite the variation of $k_5.$ We apply {\NetBiF} to this modified system to automate this result.

    In this modified network, {\NetBiF} predicts the single buffering structure given by $\Gamma = (\{X_1, X_2\}, \{R_2, R_3, R_4, R_5\}).$ Under SBA, bifurcation behavior is restricted to either $\Gamma$ or $\Gamma'=\NN\setminus\Gamma.$ Note that $\Gamma'$ is not a buffering structure, but it is a determinant structure. We visualize this partitioning below:
    $$\begin{tikzpicture}[baseline=(current  bounding  box.center), node distance = {15mm}]
            \node (0) {};
			\node (X1) [right of=0] {$X_1$};
			\node (X2) [right of=X1] {$X_2$};
			\node (X3) [right of=X2] {$X_3$};
            \node (1) [right of=X3] {};
            
            \draw [->] (0) -- (X1) node[midway,above]{$k_1$};
            
			\draw [-left to] ($(X1.east) + (0pt, 2pt)$) -- ($(X2.west) + (0pt, 2pt)$) node[midway,above]{$k_2$};
			\draw [-left to] ($(X2.west) + (0pt, -2pt)$) -- ($(X1.east) + (0pt, -2pt)$) node[midway,below]{$k_3$};

			\draw [-left to] ($(X2.east) + (0pt, 2pt)$) -- ($(X3.west) + (0pt, 2pt)$) node[midway,above]{$k_4$};
			\draw [-left to] ($(X3.west) + (0pt, -2pt)$) -- ($(X2.east) + (0pt, -2pt)$) node[midway,below]{$k_5$};

            \draw [->] (X3) -- (1) node[midway,above]{$k_6$};

            \draw[inhibit] (X2.north) 
    |- ($(X1)!0.5!(X2) + (0,30pt)$) 
    -- ($(X1)!0.5!(X2) + (0, 20pt)$)
    node[pos=0.25,left,red] {};

    \draw[activate] (X1.south) 
    |- ($(X2)!0.5!(X3) + (0,-30pt)$) 
    -- ($(X2)!0.5!(X3) + (0, -20pt)$)
    node[pos=1,left,red] {};

\draw[green!70!black,dashed,very thick,rounded corners]
  ($(X1.north west)+(0.05, 1)$) rectangle ($(X3.south west)+(-0.05,-1)$);
	\end{tikzpicture}$$ Remarkably, $X_3$ resides out of $\Gamma.$ By the Law of Localization (Theorem \ref{thm:localization}), perturbations of rate parameters inside $\Gamma$ will not affect the steady state of $X_3,$ explaining our earlier numerical observations involving $k_5.$ Now, {\NetBiF} solves the symbolic determinant of the flux-augmented matrix of the network given by\[\det \AA = \underbrace{r_{6,3}}_{\det  
    \AA_{\Gamma'}}\underbrace{(r_{2,1}r_{4,2} -r_{2,2}r_{5,1}-r_{3,2}r_{5,1})}_{\det \AA_{\Gamma}}.\] Note that $\det\AA_{\Gamma'} \ne 0$ so that $\det \AA = 0$ only when $\det \AA_\Gamma = 0.$ From this, {\NetBiF} concludes that $X_1$ and $X_2$ are bifurcation candidates whose bifurcation is dictated by the reactions $R_2, R_3, R_4$ and $R_5$ only. This behavior is restricted within the buffering structure $\Gamma.$

\end{example}
Results from these examples were obtained from {\NetBiF} in less than a second, avoiding parameter dependence and kinetic complexities. In the following subsections, we scale the implementation of {\NetBiF} to different biochemical systems showing its efficacy in biological applications.

\subsection{{\NetBiF} predicts saddle-node bifurcations across reactions in the glycolysis pathway}
The glycolysis pathway is a small, yet fundamental process that converts glucose to metabolic pyruvate through a series of biochemical reactions \citep{vanwijk2005energy}. This is part of a larger central carbon metabolic pathway that converts glucose into metabolic pyruvate that organisms use as a source of energy. The reaction network of the glycolysis pathway consists of twelve species and sixteen reactions (see Appendix \ref{appendix:rn-rep}). It is assumed that the system obeys mass action kinetics. In the same paper, three regulatory reactions were considered in the pathway, namely, the allosteric feedback activations by fructose-1,6-bisphosphate (F16BP) and fructose-2,6-bisphosphate (F26BP) and inhibition by phosphoenolpyruvate (PEP) \citep{mulukutla2014bistability}. Through SBA, the pathway has been shown to exhibit a saddle-node bifurcation when the rate parameter for the regulation of reaction $R_6$ by PEP is varied \cite{okada2021flux}.

We reproduce this result using {\NetBiF} and also show that when all other rate parameters inside the buffering structure are varied, the system indeed exhibits saddle-node bifurcation. In addition, we show that this saddle-node may be degenerate for some choices of the rate functions. Figure \ref{fig:glycolysis_network} shows the reaction network for the glycolysis pathway, along with the regulatory reactions and determinant structures of the system. 

\begin{figure}
    \centering
    \includegraphics[width=\linewidth]{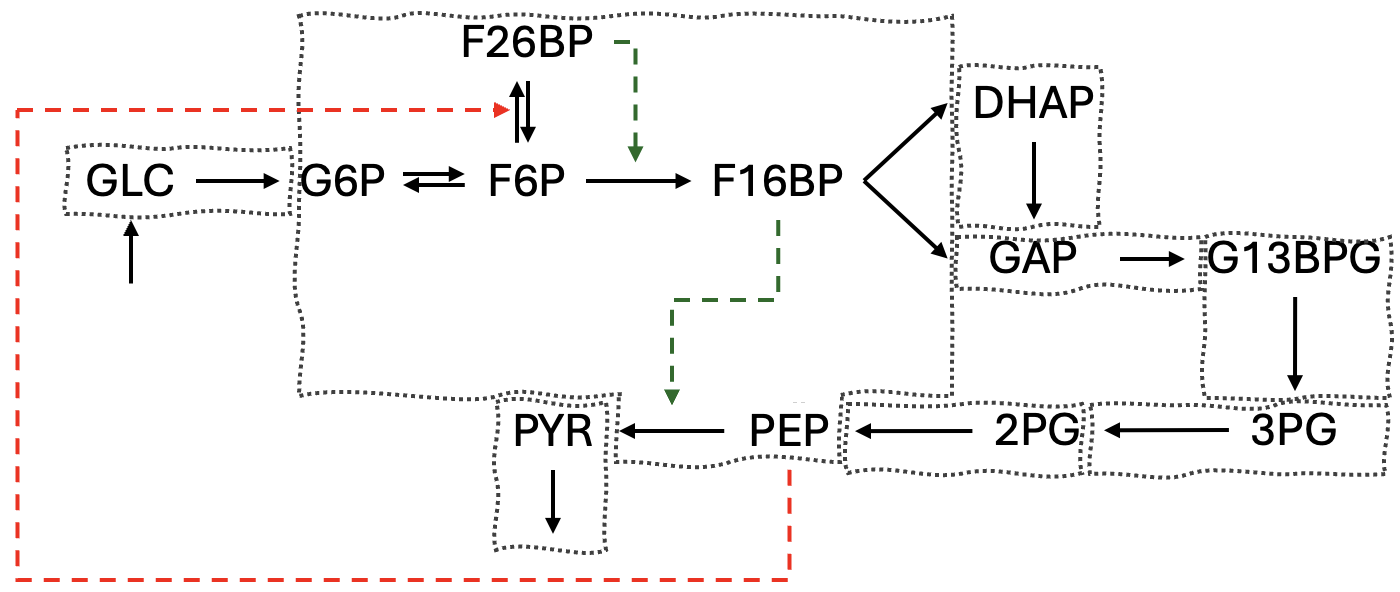}
    \caption{Network representation for the glycolysis pathway adopted from \cite{okada2021flux}. Green and red dashed arrows represent activation and inhibition regulations, respectively. Determinant structures are represented by dotted boxes.}
    \label{fig:glycolysis_network}
\end{figure}

We adopt the reaction network representation for the glycolysis pathway from \cite{okada2021flux}, where monotonic kinetics is assumed and the three allosteric regulatory reactions are considered. Using \NetBiF, we find the existence of eight determinant structures where the possibility for bifurcation may be determined. Indeed, following the algorithm, a sign change in the determinant of submatrix for $\AA$ corresponding to the largest determinant structure is observed. Specifically, this sign change is observed in the determinant structure given by $${\Gamma_8 = (\{F16BP, F26BP, F6P, G6P\}, \{R_3, R_4, R_5, R_6, R_7, R_8, R_{14}, R_{15}\}).}$$ This bifurcation behavior is found in the smallest buffering structure containing the determinant structure given by \[\begin{aligned}{\rm LBS~8} = \Gamma_8 \cup(&\{2PG, 3PG, DHAP, G13BPG, GAP, PYR\},\\
&\{R_9, R_{10}, R_{11}, R_{12}, R_{13}, R_{16}\}).\end{aligned}\] Note that ${\rm LBS~8}$ contains all species and reactions, except for glucose (GLC) and reactions $R_1$ and $R_2,$ implying that perturbations of rate parameters for reactions $R_3$ up to $R_{16}$ will not affect GLC.

Having determined the existence of bifurcation in the system within LBS 8, we classify the type of bifurcation exhibited by the system. We choose the bifurcation parameter to be the regulation of reaction $R_6$ by PEP ($k_6$). The package {\NetBiF} evaluates the net parameter sensitivity $\Theta_{R_6}$ given by \[\Theta_{R_6} = \frac{-r_{15,G6P}\cdot r_{4,F6P}\cdot r_{5,F26BP}\cdot F6P^*}{r_{3,G6P}\cdot PEP^*\cdot(r_{6,F6P}\cdot r_{5,F26BP} + r_{5,F6P}\cdot r_{7,F26BP})}\] where $F6P^*$ and $PEP^*$ denote the steady states of F6P and PEP, respectively. Since mass action kinetics is assumed so that $\partial r_n/\partial x_m > 0$ when $\nu_{m,n} < 0,$ the value of $\Theta$ is nonzero for any choice of the associated rate parameters. Then, it computes the value for (H) and identifies a sign change implying that it may be zero for some choice of rate parameters. However, this condition is parameter-rich and so the package concludes that the system exhibits a saddle-node bifurcation, which may be degenerate or a cusp when the choice of rate parameters satisfies the condition for (H) to equal zero.

Our findings support the result of \cite{okada2021flux} where they also uncovered a saddle-node bifurcation behavior across all species in $\Gamma_8$ when the rate parameter for $R_6$ is varied. Now, due to \NetBiF, we are able to do the same analysis for all of the rate parameters in the reaction set of $\Gamma_8.$ The symbolic expressions for $\Theta$ in $\Gamma_8,$ computed with respect to each reaction set as the bifurcation parameter, are all nonzero (see Appendix \ref{appendix:tables}, Table \ref{tab:theta_glycolysis}). In the context of classifying the type of bifurcation, this means that all perturbations in any of the reactions in $\Gamma_8$ are most likely to result in a saddle-node bifurcation at each of their respective bifurcation points. Still, this saddle-node bifurcation may be degenerate or result in a cusp, and an analysis of (H) for each reaction shows that this is a parameter-rich condition requiring specific and exact kinetic values.

The above results validate the findings of \cite{okada2021flux}, and extend its conclusion as we have shown that a saddle-node bifurcation is the most likely outcome when every rate parameter associated with reactions in $\Gamma_8$ is varied, and not just for $k_{6,{\rm PEP}}.$ This structural observation is consistent with the bistable switch-like behavior reported in \cite{mulukutla2014bistability}. Degeneracy of the saddle-node bifurcations for specific configurations of rate parameters is possible, which could warrant further numerical investigation. Biologically, this opens avenues for future research on this small yet fundamental pathway in energy conversion as the type of bifurcation that the system exhibits could inform how living organisms regulate switching of metabolic states given environmental and physiological perturbations.

\subsection{Structural (non-)bifurcation in the Wnt signaling pathway: Lee-Feinberg vs. Maclean model}

The Wnt signaling pathway plays a key role in the development of stem cells, and dysfunctions in the pathway have been shown to play a part in colon cancer emergence and development \citep{reya2005wnt,vermeulen2010cancer}. Moreover, bistable behavior has been observed both experimentally and theoretically for some models of the pathway \citep{cantoria2023feedback,lee2003roles,maclean2015parameter}. In this paper, we apply {\NetBiF} to investigate and compare bifurcation behaviors in two models of the Wnt signaling pathway, namely, the {Lee-Feinberg model} and the {Maclean model}. Each of the two models focuses on different aspects of the pathway. The Lee-Feinberg model focuses on the formation of the destruction complex (DC) from its constituent parts \citep{feinberg2019crnt,lee2003roles}. Meanwhile, the Maclean model studies the effects of controlling the shuttling and degradation parameters of the pathway \citep{maclean2015parameter}. We adopt the reaction network representations for the Lee-Feinberg model and Maclean et al.~model from \citep[Appendix B]{hernandez2025embedding} as input to our computational package.

First, we apply {\NetBiF} to the Lee-Feinberg model \citep{lee2003roles}, as presented in \citep[\S~10.5.4]{feinberg2019crnt} to show that its network structure excludes the possibility for bifurcation behavior. The reaction network consists of 14 species and 21 reactions (see Appendix \ref{appendix:rn-rep}). From \NetBiF, the system admits one buffering structure given by ${\rm LBS}~1 = (\{X_p\}, \{R_{14}\})$ and thus two determinant structures $\Gamma_1 = {\rm LBS}~1$ and $\Gamma_1' = \NN\setminus\Gamma_1.$ Evaluating the determinants of the flux-augmented matrix for each determinant structure, {\NetBiF} confirms that none vanish under monotonic kinetics. Indeed, we have $\det(\AA_{\Gamma_1}) = r_{14,X_p}$ and \begin{align*}\det(\AA_{\Gamma'}) = r_{18,Di}(r_{19,Da} + r_{20,Da_n})(&r_{21,P}r_{16,X_n}r_{12,CXY}r_{10,Ya_n}r_{7,T}r_{4,X}r_{1,Ya}r_{8,CXT}\\
&+ [\text{all positive terms}]).\end{align*} Since the determinant of the flux-augmented matrix for the whole network $\det\AA$ is the product of the determinants of $\AA$ for each determinant structure, there exists no values for the rate parameters for which $\det\AA = 0.$ As a result, the network structure fails to admit any candidates for steady-state bifurcations. Previous studies have shown that the Lee-Feinberg model, as represented by \cite{feinberg2019crnt}, is structurally monostationary \citep{lee2003roles}. Here, we demonstrate how {\NetBiF} supports this result by proving that the associated reaction network of the system fails to support the transitions necessary to induce multistationarity.

Now, we apply {\NetBiF} on the Maclean model, a different representation of the Wnt signaling pathway focusing on the degradation of $\beta$-catenin and shuttling of proteins between the nucleus and cytoplasm of a cell \citep{maclean2015parameter}. The original study, as well as new results from \cite{cantoria2023feedback}, has shown that perturbations on the rates of shuttling and degradation promotes bistability in the system. Here, we provide evidence of this bistability by applying structural bifurcation analysis via \NetBiF. The reaction network representation of the Maclean model from \cite{hernandez2025embedding} consists of 19 species and 31 reactions (see Appendix \ref{appendix:rn-rep}). We remark that reactions listed are not necessarily chemical reactions but rather, they can represent degradation and production reactions, or represent the shuttling dynamics of some species between the nucleus and cytoplasm. 

{\NetBiF} detects bifurcation in the system due to the presence of sign changes in the determinant of the flux-augmented matrix of the entire network. This detection was done for the entire network since {\NetBiF} finds no proper labelled buffering structure. The package tags the entire network as a trivial buffering structure. Biologically, this tells us that the entire shuttling/degradation network of Maclean acts as a singular coupled unit. This contrasts the Lee-Feinberg model where a buffering structure was found and could be decoupled from the system. This tells us that the existence (or non-existence) of buffering structures does not necessarily dictate the possibility (or non-possibility) of bifurcation behavior.

Now, we classify the type of bifurcation induced when varying the parameter responsible for the shuttling of $\beta$-catenin from the cytoplasm to the nucleus, i.e. $R_2\colon X\to X_n.$ Through \NetBiF, we see that the net parameter sensitivity is not sign definite so that it can vanish for permissible values of the rate parameters. This implies that the bifurcation point, for the bifurcation parameter $k_2,$ is a saddle-node unless
\[\begin{aligned}\Theta_{R_2} &= r_{4,X}r_{6,Xn}(r_{28,CYX} - r_{29, CYXn})
- r_{4,X}r_{31,Xn}(r_{29, CYXn}+ r_{7,CYXn})\\
&\phantom{=}+ r_{30,X}r_{6,Xn}(r_{28, CYX} + r_{5,CYX}) = 0.\end{aligned}\] Note that only the first term, $r_{4,X}r_{6,Xn}(r_{28, CYX} - r_{29, CYXn})$ is sign indefinite where the difference of $r_{28,CYX}$ and $r_{29, CYXn}$ represents the imbalance between the dissociation rates of Dsh ($Y$) and $\beta$-catenin in the nucleus ($X_n$) and cytoplasm ($X$). Indeed, this term is positive when the destruction-$\beta$-catenin complex in the cytoplasm dissociates much quicker than that in the nucleus. The two other terms in $\Theta$ are sign definite and represent the contributions of nuclear degradation and dissociation of the Dsh $\beta$-catenin complex ($R_{29}\colon CYX_n\to Y_{an}$, $R_7\colon CYX_n\to Y_{an} + X_n$) and cytoplasmic degradation and dissociation ($R_{28}\colon CYX\to Y_a$, $R_5\colon CYX\to Y_a + X$). The condition $\Theta = 0,$ where the bifurcation is not a saddle-node, therefore requires the precise balance of these three terms. In particular, if the cytoplasmic and nuclear dissociation of the Dsh $\beta$-catenin complex are the same ($k_{28, CYX} = k_{29, CYX_n}$), the bifurcation reduces to a balance of the remaining two terms: \[ r_{4,X}r_{31,Xn}(r_{29, CYXn}+ r_{7,CYXn}) = r_{30,X}r_{6,Xn}(r_{28, CYX} + r_{5,CYX}).\] Now, if $\Theta \ne 0$ so that we have a saddle-node, {\NetBiF} finds that the Hessian contraction at the bifurcation point is also sign indefinite. Therefore, the reaction network structure does not rule out degenerate saddle-nodes or cusp bifurcations in the system.

Our results tell us that the qualitative transitions of $\beta$-catenin levels through a saddle-node bifurcation is not necessarily a result of reaction network structure \citep{maclean2015parameter,cantoria2023feedback}. However, through the symbolic expressions generated by {\NetBiF} for the net parameter sensitivity, we have shown that perturbation in the shuttling, as well as degradation rates, of $\beta$-catenin is sufficient to induce bistable behavior. We remark that our analysis through {\NetBiF} only required information on network structure and the assumption of power law kinetics, without knowing the actual parameter values. Ultimately, {\NetBiF} effectively transforms a previously parameter-dependent analysis into an accessible and structure-based approach through SBA, which researchers can easily apply to more complex biochemical systems.

\subsection{Structural changes in eukaryotic heatshock response modifies bifurcation behavior}

The heat shock response of cultured cells is a universal feedback mechanism upon a cell's exposure to elevated temperatures. It acts by releasing a small number of heat shock proteins (hsp) which have been observed to provide protection to the cell, and by extension, the organism \citep{lindquist1988heat}. Understanding this stress response of eukaryotic cells is proving to be a worthwhile research endeavor as the heat shock response has also been implicated in the control of cell growth and activated even in the absence of direct physiological stress \citep{jolly2000role,sarge1993activation}. Although the heat shock response model exhibits highly nonlinear dynamics, the work of \cite{petre2011HSR} proposes a simple mass action model for the eukaryotic heat shock response. The same model was validated using existing literature and actual experiments. We add to the growing study of the heat shock response by looking at its capability for bifurcation behavior using the proposed model of \cite{petre2011HSR}. Our results can capture emergent behaviors such as bistability, hysteresis, and threshold-driven switch-like behavior in the system.

We implement SBA using {\NetBiF} on the mass action kinetics representations of the heat shock response model on eukaryotic cells from \cite{petre2011HSR}. Specifically, we implement the complete {\NetBiF} pipeline to the reduced, full, and extended models in order to show how structural modifications to a system's reaction network is sufficient to affect the existence of buffering structures, and consequently, its capacity to admit steady-state bifurcations. 

First, we look at the complete heat shock response model from \cite{petre2011HSR} to confirm the possibility of bifurcation using its structure. The reaction network consists of 10 species and 17 reactions (see Appendix \ref{appendix:rn-rep}). In this case, {\NetBiF} identifies no proper labelled buffering structure in the complete model and proceeds to compute for the determinant of the flux-augmented matrix of the entire reaction network. In this case, the package flags a sign change in the determinant of the flux-augmented matrix. Therefore, the system is capable of exhibiting bifurcation through any one of its parameter values. Next, we analyze the reduced model, which removes four reactions, namely, the reverse dissociation reactions of HSF dimers to monomers ($R_2\colon D\to 2F$), reverse dissociation of HSF trimers into dimers and monomers ($R_4\colon T\to F + D$), the reverse dissociation of HSF trimers ($R_6\colon TE\to T + E$), and the disassembly of HSF trimers as mediated by HSP ($R_{11}\colon P+T\to PF + 2F$). The resulting network still involves all species in the full network, but now consists of 13 reactions only. It was shown in \cite{petre2011HSR} that the removal of these reactions does not alter the fundamental dynamical characteristics of the system. This is supported by the fact that {\NetBiF} was able to find that its reaction network representation still does not admit any proper labelled buffering structure, and also exhibits sign changes in the determinant of the flux-augmented matrix. Therefore, similar to the full model, the reduced model of the eukaryotic heat shock response is structurally capable of exhibiting bifurcation behavior.

Finally, similar to the previous models, the extended heat shock response model exhibits bifurcation behavior, as structurally identified by \NetBiF. This reaction network representation of this model consists of 14 species and 25 reactions, incorporating misfolding and refolding of the heat shock protein and heat shock factors (see Appendix \ref{appendix:rn-rep}). In the network, this is represented by six additional reactions, acting as feedback loops in the system. In contrast to the previous models which does not admit any proper labelled buffering structures, the package {\NetBiF} identifies three determinant structures in the extended model from which bifurcation may arise, namely, $\Gamma_1\colon (\{Y\}, \{R_{20}, R_{21}\}), \Gamma_2\colon (\{PY\}, \{R_{18}, R_{24}\}),$ and $\Gamma_3\colon \NN\setminus (\Gamma_1\cup \Gamma_2).$ Moreover, the package tags $\Gamma_1$ and $\Gamma_1\cup \Gamma_2$ as proper labelled buffering structures, from which  perturbations of rate parameters within such structures remain inside them. Now, in the analysis of the signs of the determinant of submatrices of $\AA,$ the sign change occurs in $\Gamma_3$ only. Since no proper labelled buffering structure contains all of the species found in this determinant structure, the package concludes that bifurcation behavior is possible through perturbations of any of the rate parameters in the system. 

The results of the structural bifurcation analysis of the full, reduced, and extended models of the heat shock response show a consistent outcome that bifurcation is possible upon selection of appropriate kinetics and rate parameters. These conclusions were obtained using information on the structure of the systems' reaction networks. Now, the models of \cite{petre2011HSR} work under the assumption of mass action kinetics, which have been validated using existing literature and their own experimental evidence. Since the assumption is mass action kinetics, analysis of such models may be performed using existing modules such as the CRNToolbox (\citeyear{CRNToolbox}). Table \ref{tab:CRNToolboxHSR} shows some reports from the CRN Toolbox for the heat shock response. We remark that only the extended model is capable of exhibiting multistationarity under the assumption of mass action kinetics. Nonetheless, all three models exhibit network discordance so that there exists a weakly monotonic kinetics for which the resulting kinetic system is not injective. As a result, in such kinetics, the system is able to exhibit at least two positive steady states.

\begin{table*}[h]
    \centering
    \caption{Some CRN Toolbox reports for the full, reduced, and extended models of the eukaryotic heat shock response. For the multistationarity property, the toolbox assumes mass action as kinetics for each network.}
    \begin{tabular}{lccc}
        \hline Property & Reduced & Full & Extended \\\hline
        Deficiency zero & No & No & No \\
        Weak reversibility & Yes & Yes & Yes \\
        Concordance & No & No & No\\
        Injectivity & No & No & No\\
        Multistationarity & No & No & Yes\\\hline
    \end{tabular}\label{tab:CRNToolboxHSR}
\end{table*}

The complementary results of {\NetBiF} and the CRNToolbox show that the presented computational package herein should not be considered as replacements for existing tools. Rather, the package {\NetBiF} acts as a structural diagnostic screen which is able to make conclusions on the dynamics of the system but without any assumption on the kinetics endowed. Indeed, the structural analysis is able to anticipate the possibility of bifurcation, and thus possibly, multistationarity, which the CRN Toolbox recovers but with the additional assumption of mass action kinetics. The results presented for the eukaryotic heat shock response models underline the importance of performing diagnostic screens from the network structure prior to and independently of any specific kinetics.

\subsection{Structural bifurcation analysis of carbon cycle with direct ocean capture complements CRNT approaches to establish multistationarity}

We explore the bifurcation behavior in a carbon cycle model with direct ocean capture (DOC) technology. Conditions for multistationarity, which can signal crucial tipping points in the carbon cycle, have already been explored in a reaction network based analysis of the DOC model \citep{alamin2026DOC}. Moreover, in the same paper, conditions for the existence of positive steady states regardless of kinetic parameters have been established. Determining when a carbon cycle model exhibits bifurcation is essential as it implies the possibility for crucial tipping points. These irreversible and catastrophic changes in the carbon cycle are signaled by the existence of bistable branches in the system \citep{anderies2013topology}. Here, we analyze the bifurcation behavior of the system using its network structure and topology.

We use the reaction network representation for the DOC model used in \cite{alamin2026DOC}. The reaction network consists of five species and seven reactions. The species represent the major carbon pools in the system, namely, the land biota ({\rm Land}), ocean ({\rm Oce}), atmosphere ({\rm Atm}), long term storage ({\rm LTS}), and the DOC technology ({\rm DOC}), which captures carbon directly from the ocean. We consider the reaction network representation of the DOC system in \eqref{eq:DOC-system}. 
\begin{equation}
\label{eq:DOC-system}
    \begin{cases}
    \begin{aligned}
    &R_1\colon{\rm Land} + 2 \ {\rm Atm} \to 2 \ {\rm Land} + {\rm Atm}
    &&R_5\colon{\rm LTS}\to {\rm Atm}\\
    &R_2\colon 2 \ {\rm Land} + {\rm Atm}\to {\rm Land} + 2 \ {\rm Atm}
    &&R_6\colon{\rm DOC}\to {\rm LTS}\\
    &R_3\colon{\rm Atm} \to {\rm Oce}&&R_7\colon{\rm Oce}\to{\rm DOC}\\
    &R_4\colon{\rm Oce} \to {\rm Atm}
    \end{aligned}
\end{cases}
\end{equation}

In this case, {\NetBiF} finds that the network does not admit a proper labelled buffering structure and tags the entire model as a trivial buffering structure. The determinant of the flux-augmented matrix in its factored form is given by
\[\begin{aligned}
\det \AA &= r_{6,{\rm DOC}}\,r_{5,{\rm LTS}}\,(r_{4,{\rm Oce}} + r_{7,{\rm Oce}})\,(r_{1,{\rm Atm}} - r_{2,{\rm Atm}}) \\
&\phantom{=} - (r_{1,{\rm Land}} - r_{2,{\rm Land}})
\Bigl[
r_{6,{\rm DOC}}\,r_{5,{\rm LTS}}\,(r_{4,{\rm Oce}} + r_{7,{\rm Oce}} + r_{3,{\rm Atm}}) \\
&\phantom{= -(r_{1,{\rm Land}} - r_{2,{\rm Land}})\Bigl[}
  + r_{3,{\rm Atm}}\,r_{7,{\rm Oce}}\,(r_{6,{\rm DOC}} + r_{5,{\rm LTS}})
\Bigr].
\end{aligned}\] From here, we see that the capability of $\det\AA$ to vanish relies on the values of $p:= r_{1,{\rm Land}} - r_{2,{\rm Land}}$ and $q:= r_{1, {\rm Atm}} - r_{2,{\rm Atm}},$ i.e. the differences of the interactions of photosynthesis and respiration in land and atmosphere. 

It was shown in \cite{alamin2026DOC} that if both $p = 0$ and $q = 0,$ then the system can only exhibit positive steady states if the rate constants for reactions $R_1$ and $R_2$ coincide. This conclusion was obtained using the notion of $\widehat T$-independence of the system's network decomposition \citep{alamin2024positive}. Now, if we assume that $p = 0$ or $q = 0,$ but not both, then $\det\AA$ will never vanish for any value of the rate parameters. Therefore, no bifurcation behavior is possible if either $p = 0$ or $q = 0,$ but not both. This conclusion supports the findings of Alamin et al.~\cite[Table 3]{alamin2026DOC} where such systems (called $P$-null and $Q$-null) have been shown to be monostationary. Consequently, if neither $p$ nor $q$ equals zero, then the DOC system is capable of exhibiting bifurcation since $\det\AA$ can vanish for some specific choice of rate parameters. Again, the same properties were observed in the analysis of these classes of ``positive" and ``negative" DOC systems (see \cite{alamin2026DOC} for more details).

Now, under the assumption that $p\neq 0$ and $q\ne 0,$ we analyze the type of bifurcation that the system exhibits using the conditions in Table \ref{tab:okada} which uses information on network structure alone. We present in Appendix \ref{appendix:tables} Table \ref{tab:theta_DOC} the values of the net parameter sensitivity for each of the reactions.

We remark that the symbolic expressions for $\Theta$ in all reactions, except for $R_7$ in the DOC model are structurally nonzero. This means that perturbations in the rate parameters of these reactions will most likely result in a saddle-node bifurcation at the bifurcation point. Still, under a specific balance of rate parameters, the saddle-node bifurcation may be degenerate. Meanwhile, for perturbations in the rate parameters of reaction $R_7,$ the system most likely exhibits a saddle-node bifurcation except when \[r_{4,{\rm Oce}}(r_{6,{\rm DOC}} + r_{5,{\rm LTS}}) - r_{6,{\rm DOC}}r_{5,{\rm LTS}} = 0 \iff r_{4,{\rm Oce}} = \frac{r_{6,{\rm DOC}}r_{5,{\rm LTS}}}{r_{6,{\rm DOC}} + r_{5,{\rm LTS}}}.\] In this scenario, $r_{4,{\rm Oce}}$ is in perfect balance with a multiple of the harmonic mean of $r_{6,{\rm DOC}}$ and $r_{5,{\rm LTS}}$ at the bifurcation point. From a biological perspective, we have emphasized the possibilities for tipping points in the carbon cycle and especially in perturbations of the parameter in $R_7\colon{\rm Oce}\to {\rm DOC},$ these tipping points are governed by a precise balance of $r_{4,{\rm Oce}},$ representing removal of carbon from ocean towards atmosphere, $r_{5,{\rm LTS}},$ representing removal of carbon from long term storage, and $r_{6,{\rm DOC}},$ representing efficacy of DOC system to store captured carbon to long term storage. These findings demonstrate the capability of {\NetBiF} to extend beyond determination of bifurcation candidates, but the package also yields biologically interpretable conditions due to its symbolic calculation of relevant expressions. Taken together, we have shown how applying SBA in the DOC model through {\NetBiF} can inform engineering designs and limitations of carbon sequestration strategies.

\subsection{Reaction network structure of positive feedback systems induce transcritical bifurcations}

Positive feedback loops in gene regulation, where the protein promotes its own synthesis, frequently occur in different biological systems \citep{angeli2004detection,becskei2001positive,becskei2000engineering,ferrell2002self,harris2005p53,suel2006excitable,surget2013uncovering}. Here, we examine a simple single-loop autoregulation dynamical model of gene regulation whose biochemical reaction network involves four species and six reactions \citep{muller2009transcritical}. The structure of this model matches closely with prior studies involving a positive feedback loop leading to cellular differentiation \citep{becskei2001positive,becskei2000engineering,suel2006excitable}. We discuss this class of biological systems to show that although most steady state bifurcations in biochemical systems result in saddle-node bifurcations, {\NetBiF} is still able to recover transcritical or pitchfork bifurcations through the system's network structure.

We consider the reaction network representation of autoregulatory gene expression with positive feedback in \eqref{eq:positive-loop}.
\begin{equation}
    \label{eq:positive-loop}
    \begin{cases}
    \begin{aligned}
        &R_1\colon PMP\to PMP + mRNA
        &&R_2\colon mRNA\to \emptyset\\
        &R_3\colon mRNA\to mRNA + P
        &&R_4\colon P\to \emptyset\\
        &R_5\colon MP + P \to PMP
        &&R_6\colon PMP\to P + MP
        \end{aligned}
    \end{cases}.
\end{equation} 
In the model, it is assumed that there is a difference in time scales between the formation of the protein/mRNA polymerase complex (reactions $R_5$ and $R_6$) and the other processes. In \cite{muller2009transcritical}, they showed that the system exhibits a steady state bifurcation at the origin whenever the parameter describing the regulation strength of the autoregulatory loop exceeds some threshold value. Moreover, it was found that this bifurcation is a transcritical bifurcation wherein the origin becomes unstable and a second stable nonzero equilibrium emerges. Here, we recover the same result via {\NetBiF} using only information on the network structure.

{\NetBiF} finds no proper labelled buffering structure and computes the determinant of the flux-augmented matrix for the entire network. Bifurcation is structurally observed whenever \[
\begin{aligned}&r_{4,P}r_{5,MP}r_{2,mRNA} - r_{5,P}r_{1,PMP}r_{3,mRNA} + r_{4,P}r_{6,PMP}r_{2,mRNA} = 0\\
&\iff  r_{4,P}r_{2,mRNA}(r_{5,MP} + r_{6, PMP}) = r_{5,P}r_{1,PMP}r_{3,mRNA}.\end{aligned}\] We show that the origin exhibits a transcritical bifurcation whenever the rate parameter for reaction $R_6$ is perturbed near some threshold value. The net parameter sensitivity for reaction $R_6$ in the model is given by \[\Theta = \frac{(MP^*)(P^*)}{r_{5,MP}} = 0\] since the steady state is the origin. Since the reaction network is small, we can write the values corresponding to (A1), (H), and (T3) here. We have \[\text{(A1)} = \frac{1}{r_{5,MP}r_{6,PMP}},\quad\text{(H)} = \frac{-2(r_{5,MP} + r_{6,PMP})}{r_{5,P}r_{6,PMP}^2\cdot P^*},\quad\text{(T3)} = 0.\] Since (A1) $\ne 0$ and (H) $\ne 0$, the bifurcation classification criteria (see Figure \ref{fig:classify_flowchart}) provide that the system exhibits a transcritical bifurcation at the origin.

We remark that the code for {\NetBiF} reports a saddle-node bifurcation since its formulation assumes that the bifurcations analyzed are of interior steady states. This follows from the assumption of monotonic kinetics (e.g. power law kinetics) so that transcritical and pitchfork bifurcations are excluded for interior steady states. Transcritical bifurcations usually require a persistent zero-state equilibrium, as in our example \cite{strogatz2018nonlinear}. Nonetheless, the values computed and conclusions remain valid using the results of \cite{okada2021flux}. 

Unlike our previous examples which only show saddle-node bifurcations which may be degenerate or cusps, the discussion here shows that {\NetBiF} is able to recover other types of steady-state bifurcations when accompanied with manual analyses. Under the assumption of monotonic kinetics, we must enforce a persistent equilibrium at the origin to conclude, for example, a transcritical bifurcation.

\subsection{A remark on pitchfork bifurcations}

We close the results by discussing the case of pitchfork bifurcations. True pitchfork bifurcations are rare in biochemical systems as they require an exact symmetry in the rate parameters. This precise mathematical requirement is necessary but biologically fragile \citep{golubitsky1983discussion,widder2007dynamic}. To bypass these precise conditions, we examine an idealized two-gene cross-regulation network with cooperative repression \citep{widder2007dynamic}. Such systems have particular applications in modelling genetic switches.

The reaction network representation of the model consists of eight species and 12 reactions given in \eqref{eq:widder-model}. The first two reactions correspond to the cooperative binding/unbinding of proteins ($P$), genes ($G$), and gene-protein complexes ($H$). The remaining reactions correspond to the transcription, translation, and degradation of genes and proteins.

\begin{equation}
    \label{eq:widder-model}
    \begin{cases}
        \begin{aligned}
        &R_1\colon G_1 + 2P_2 \to H_1 \qquad && R_7\colon Q_1\to Q_1+P_1\\
        &R_2\colon H_1\to G_1 + 2P_2 && R_8\colon Q_2\to Q_2+P_2\\
        &R_3\colon G_2 + 2P_1\to H_2 && R_9\colon Q_1\to \emptyset\\
        &R_4\colon H_2\to G_2 + 2P_1 && R_{10}\colon Q_2\to \emptyset\\
        &R_5\colon G_1\to G_1 + Q_1 && R_{11}\colon P_1\to\emptyset\\
        &R_6\colon G_2\to G_2 + Q_2 && R_{12}\colon P_2\to\emptyset\\
        \end{aligned}
    \end{cases}.
\end{equation}

Using \NetBiF, we find that there are no properly labeled buffering structures, and the determinant of the flux-augmented matrix is computed as a whole for the network. Still, {\NetBiF} discovers a unique bifurcation condition through sign changes in the symbolic expression of the determinant of the flux-augmented matrix.

We classify the bifurcation in the system when the rate parameters for $R_5$ and $R_6$ are perturbed. From \cite[see Table 3 and Figure 10]{widder2007dynamic} , the bifurcation is pitchfork whenever there is exact symmetry between the values of the parameters. In the case of our application, we specify the reactions containing the bifurcation parameter to both be $R_5$ and $R_6.$ {\NetBiF} giving the net parameter sensitivity of the model to be
\[\Theta = G_2^*\frac{r_{3,G_2} + r_{4,H_2}}{r_{4,H_2}r_{6,G_2}} - G_1^*\frac{r_{3,P_1}r_{7,Q_1}}{r_{11,P_1}r_{4,H_{2}}r_{9,Q_1}}.\] Note that $\Theta = 0$ for some appropriate choice of the rate parameters. Meanwhile, the symbolic expression corresponding to (A1) is given by \[\text{(A1)} = \frac{-r_{1,P_2}r_{3,P_1}r_{7,Q_1}(r_{5,G_1} + r_{6,G_2})}{r_{11,P_1}r_{12,P_2}r_{4,H_2}r_{6,G_2}r_{9,Q_1}(r_{1,G_1} + r_{2,H_1})},\] which is nonzero for any permissible choice of rate parameters. Through \NetBiF, the symbolic expressions for the Hessian contraction (H) can equal zero for some choice of rate parameters, while that of the third derivative contraction (T3) is always nonzero. Using the decision tree for the classification of the bifurcation (see Figure \ref{fig:classify_flowchart}), the package concludes that the system is structurally capable of exhibiting a pitchfork bifurcation given the choice of perturbed rate parameters. 

In any case, we remark that our presented example is for demonstration purposes only to show that network structure does not preclude the possibility for a pitchfork bifurcation. In most cases, the observed bifurcation remains a saddle-node as the exact symmetry of the rate parameters is rarely achieved \citep{widder2007dynamic}. Our application here demonstrates how {\NetBiF} is able to also capture pitchfork bifurcations, alongside the usual saddle-node and transcritical bifurcations, underscoring its utility in analyzing steady state bifurcations of biochemical systems.

\section{Summary and Outlook}\label{sec:summary}

The long-term dynamical behavior of biochemical systems is governed by the structure of their underlying reaction networks. Classical techniques to analyze bifurcation behavior rely heavily on explicit rate parameters, which are often unavailable for large experimental systems. We address this gap by implementing the method of Structural Bifurcation Analysis (SBA) through a computational package, {\NetBiF} (\textbf{Net}work \textbf{Bi}furcation \textbf{F}inder), which leverages the concepts of the flux-augmented matrix, buffering structures, and the law of localization.

In this study, we demonstrated the efficacy and biological relevance of {\NetBiF} as a computational tool across different biochemical systems. Specifically, bistable switch-like behaviors in the glycolysis pathway were verified by {\NetBiF} by recovering the bistable dynamics in the glycolysis pathway. Notably, through the computational efficiency of {\NetBiF}, the structural persistence of saddle-node bifurcations across all relevant reactions was predicted, generalizing the previous result analyzing only one regulatory parameter. For the Wnt signaling pathway, {\NetBiF} was utilized to structurally preclude bifurcation in the Lee-Feinberg model and explore the bistable potential of the Maclean shuttling model. We have also shown how structural modifications in models of eukaryotic heat shock response dictate the existence of buffering structures and subsequent bifurcation behaviors. Our analysis also extended to a large-scale carbon cycle model with direct ocean capture, providing parameter-independent conditions for saddle-node bifurcations. Beyond identifying these bifurcation behaviors, we have established that {\NetBiF} can effectively classify some qualitative changes into saddle-node, transcritical, or pitchfork bifurcations automating structural conditions derived from Sotomayor's theorem for systems with power law kinetics.

The application of {\NetBiF} to these systems complements existing methodologies, such as CRNToolbox, COMPILES, and CRITERIA. Nonetheless, the current framework still has limitations that outline some avenues for future research. First, future iterations of the software may be expanded to detect emergence of limit cycles through Hopf bifurcation by considering more structure-based correspondence of the Jacobian and the flux-augmented matrix. Second, while the determination of bifurcation candidates in {\NetBiF} is largely independent of kinetics, the classification is specialized for systems with power law kinetics only. Extending the classification module to handle broader classes of monotonic kinetics could widen the applicability of the tool. Finally, as the size of biochemical systems inherently grows, optimizing the symbolic computation of the flux-augmented matrix benefits improvement of the computational efficiency.

The advancement of theory and methods in structural bifurcation analysis supports the paradigm shift of studying dynamical behaviors from a parameter-centric approach to a structure-based one. The computational package \NetBiF, through its implementation and application, serves as an important step in this direction to allow researchers to retrieve biological insights solely from the structure of the reaction network of biochemical systems.

\section*{Appendix}
\appendix
\section{Chemical Reaction Network Theory}\label{appendix:crnt}
We present the basic definitions and notions of CRNT here. 

\begin{definition}
    A \emph{reaction network} is a triple $\NN = (\SS, \CC, \RR)$ of nonempty finite sets, namely,
    \begin{enumerate}
        \item a set of \emph{species} in $\SS = \{x_1, x_2, \ldots, {x_M}\},$
        \item a set of \emph{complexes} in $\CC = \{y_1, y_2,\ldots, {y_{|\CC|}}\}\subseteq{\rm span}\ \SS$ where each $y\in\CC$ is a nonnegative linear combination of species, and
        \item {a set} of \emph{reactions} in $\RR = \{r_1, r_2, \ldots, {r_N}\}\subseteq\CC\times \CC$ where each $r = (y,y')\in\RR$ is seen as a directed edge from the \emph{reactant complex} $y$ to a \emph{product complex} $y'.$ We write a reaction $(y,y')$ as $y\to y'.$
    \end{enumerate}
\end{definition}

In addition, we also assume that no complex in a reaction network reacts to itself (i.e. $y \neq y'$ for each $y\to y'\in\RR$). Nonetheless, each complex must participate in a reaction as either a reactant or a product. That is, we have for every $y\in \CC,$ there is a $y'\in \CC$ such that either $y\to y'\in\RR$ or $y'\to y\in\RR.$

Each complex $y\in\CC$ is written as {\(y = a_1x_1 + a_2x_2+\cdots + a_M x_M = \sum_{i=1}^Ma_ix_i,\) where each of the nonnegative coefficients $a_i$'s in the linear combination is called the \emph{stoichiometric coefficient} of species $x_i$ in $y.$} If needed, we use $\emptyset\in\CC$ to denote the \emph{zero complex,} defined as the complex where all stoichiometric coefficients equal zero (i.e. $\emptyset := \sum_{i=1}^{M} 0x_i.$) A reaction involving the zero complex, for example $\emptyset\to y$ or $y\to\emptyset$ should not be interpreted as a spontaneous formation (or destruction) of $y$ from (or to) nothing, but rather that $y$ is produced from or degrades into an external reservoir.

The \emph{reaction vector} for a reaction $y\to y'$ is defined by the difference $y' - y$. These reaction vectors induce a linear subspace $S$ called the \emph{stoichiometric subspace} given by $S = \mathrm{span}\{y'-y \in \mathbb{R}^{{M}} \mid y \rightarrow y' \in \mathcal{R}\}.$ The \emph{stoichiometric matrix} $\nu$ of a network is the ${M\times N}$ matrix whose columns correspond to the coefficients of the associated species (called \emph{stoichiometric coefficient}) for each reaction vector. Note that, by construction, the dimension of the stoichiometric subspace of a network is the same as the rank of its associated stoichiometric matrix.

The stoichiometry of a network describes the relationship among the different species taking part in the different reactions. However, this information is not enough to determine the long term behavior and describe the dynamics of these species. For this, we endow our network with a rate function which describes the formation and evolution of its species over time.

\begin{definition}
    \label{def:kinetics}
    A \emph{kinetics} $\mathcal K$ of a reaction network $\NN = (\SS, \CC, \RR)$ is an assignment of differentiable \emph{rate functions} $r_{y\to y'}\colon {\mathbb R_{>0}^{M}\times \mathbb R^N}\to \mathbb{R}_{{\ge 0}}$ to each reaction $y\to y'\in\RR$ such that for {each} $\tilde y\in \mathbb R_{>0}^M,$ we have \[r_{y\to y'}(\tilde y; k) > 0 \text{ if and only if } {\rm supp}\ y \subset {\rm supp}\ \tilde y,\] where {${\rm supp}\ y$} denotes the support of the vector $y,$ i.e.~the set of species with nonzero coefficients in $y$ and $\mathbf k = [k_1\ \ k_2\ \ \ldots\ \ k_r]\in {\mathbb R^N_{>0}}$ denotes the \emph{{rate} parameters} of our reactions.
\end{definition}

For clarity, the relation ${\rm supp}\ y\subset {\rm supp} \ \tilde y$ means that every species present in the complex $y$ must also be present in the complex $\tilde y.$ Now, we define a \emph{kinetic system,} denoted $(\NN, \mathcal K)$ as a reaction network $\NN$ endowed with some kinetics $\mathcal K.$ A kinetic system defines a dynamical system of ordinary differential equations, called the \emph{species formation rate function,} describing the evolution of each species in a network over time. Hence, we get
\begin{equation}
\label{eq:ode-system}
\frac{\dd x_i}{\dd t} = \sum_{n=1}^N \nu_{i, n}r_n( x; \mathbf k) =: f_i( x; \mathbf k)\qquad  i = 1,\ldots, M,\end{equation} where $x_i$ is the concentration of species $X_i$ in the system. The points $x^*\in \mathbb R^M$ satisfying $f(x^*;\mathbf k) = 0,$ as defined in  \eqref{eq:ode-system}, are called the \emph{steady states} (or \emph{equilibria}) of the associated kinetic system. Specifically, if each element in the steady state $x^*$ is positive so that $x^*\in\mathbb R^M_{>0}$ and $f(x^*;\mathbf k) = 0,$ then $x^*$ is said to be a positive steady state of the kinetic system.

As an example, a kinetic system is said to obey the \textit{power law} if for each reaction $i\in\RR,$ there exists a positive number $k_i$ such that for any $x\in\mathbb R_{>0}^M,$ we have \[r_i(x) = k_i\prod_{j=1}^M x_j^{F_{ij}} := k_ix^{F_i}\] where $F=[F_{ij}]\in\mathbb R^{N\times M}$ is the kinetic order matrix containing the kinetic order values of each reaction. Moreover, if the column vectors of the kinetic order matrix contain the stoichiometric coefficients of the species in the reactant complex, then the system is said to obey the familiar \textit{law of mass action.}

\section{Proof of Theorem 2}\label{appendix:proof-thm-2}

\begin{theorem}
    Let $\Gamma$ be the minimal buffering structure containing the bifurcation condition at steady state $x^*$ and bifurcation parameter $\overline{\mathbf k}$, and let $\AA_{\Gamma}$ be the restriction of the flux-augmented matrix to $\Gamma.$ Assume $A_{\Gamma}$ has a simple zero eigenvalue, with left and right eigenvectors $W$ and $U = [u^\top, v^\top],$ respectively.

    If the system is endowed with power-law kinetics such that the rate function for each reaction $R_i\in\RR$ is $r_i(x) = k_i\prod_{j=1}^M x_j^{F_{ij}}$ with $F_{ij}$ being the kinetic order value of species $X_j$ in reaction $R_i,$ then the bifurcation classification quantities at steady state $x = x^*$ take the following forms:
    \begin{description}
    \item[(1) Net Sensitivity ($\Theta$):] 
    \[ \Theta = W^\top \frac{\partial r}{\partial k_n}{\bigg |}_{x = x^*} = W_n\prod_{\ell = 1}^M (x_\ell^*)^{F_{n\ell}} \] Note that for positive steady states, $\Theta = 0$ if and only if $W_n =0 .$    
    \item[(2) Parameter Sensitivity Gradient (A1):] 
    \[ W_n\left[\nabla_x\left(\frac{\partial r_n}{\partial k_n}\right)\cdot u\right] = W_n\frac{\partial r_n}{\partial k_n}{\bigg |}_{x=x^*}\cdot F_n^\top \tilde u,\qquad \tilde u = (u_i/x_i^*)_{i=1}^M\]
    
    \item[(3) Hessian Contraction (H):]
    \[ W^\top\nabla^2_x r(x^*;\mathbf{\overline k})(u,u) = \sum_{i} W_i r_i(x^*)\left((F_i^\top\cdot \tilde u)^2 - \sum_jF_{ij}\tilde u_j^2\right) \]
    
    \item[(4) Third-Order Contraction (T3):]
    \[ W^\top \nabla_x^3 r(x^*;\overline{\mathbf k})(u,u,u) = \sum_{i} W_i r_i(x^*)\left[(F_i^\top\cdot \tilde u)^3 - 3(F_i^\top\cdot \tilde u)\sum_jF_{ij}\tilde u_j^2 + 2\sum_jF_{ij}\tilde u_j^3\right] \]
\end{description}
\end{theorem}

\begin{proof} The first quantity, denoted $\Theta$ is the net of the system projected onto the left eigenvector of $\AA_\Gamma$ at the steady state $x^*:$ \[\Theta = W^\top \frac{\partial r}{\partial k_n}{\bigg |}_{x = x^*}.\] Under power law kinetics, the rate function for each reaction $R_n\in\RR$ is $r_n(x) = k_n\prod_{\ell=1}^M x_\ell^{F_{n\ell}}$ where $F_{n\ell}$ is the kinetic order value of species $X_\ell$ in reaction $R_n.$ Moreover, since only reaction $R_n$ depends on rate parameter $k_n,$ only the $n$th coordinate of $\Theta$ is nonzero so that $$\Theta = W_n\underbrace{\prod_{\ell = 1}^M(x_{\ell})^{F_{n\ell}}}_{\partial r/\partial k_n}$$. Note that $\partial r/\partial k_n > 0$ whenever $x^*$ is a positive steady state. This proves the statement that for positive steady states, we have $\Theta = 0$ if and only if $W_n = 0.$ As such, we exclude the possibility for transcritical and pitchfork bifurcations of positive steady states, and consider these bifurcations only when it occurs at the boundary or the origin. For positive steady states, the classification is heavily reliant on higher-order derivatives.

Next, the algorithm computes the quantity (A1). This quantity measures how the parameter sensitivity $\partial r/\partial k_n$ changes along the direction of $u.$ For positive steady states, the value of (A1) is irrelevant to our analysis. In contrast, if $\Theta = 0$ so that the bifurcation occurs at the boundary or the origin, the value of (A1) determines if the system is structurally capable of exhibiting a steady-state bifurcation. 

For the calculation of (A1) under power law kinetics, we have \[\frac{\partial r_n}{\partial k_n} = \prod_{\ell=1}^M x_\ell^{F_{n\ell}}\implies \frac{\partial^2r_n}{\partial k_n\partial x_i}{\bigg |}_{x=x^*} = \frac{\partial^2 r_n}{\partial x_i\partial k_n}{\bigg |}_{x=x^*} = \frac{F_{ni}}{x_i^*}\underbrace{\prod_{\ell=1}^M(x_\ell^*)^{F_{n\ell}}}_{{\partial r_n}/{\partial k_n}}.\] Then, the gradient with respect to $x$ becomes the accumulation of each of the mixed partial derivatives against each reaction. Hence, for the quantity (A1), we have \[W_n\left[\nabla_x\left(\frac{\partial r}{\partial k_n}\right)\cdot u\right] = W_n\frac{\partial r}{\partial k_n}{\bigg |}_{x=x^*}\cdot F_n^\top \tilde u,\] where $\tilde u = u_i/(x_i)^*.$  To further distinguish the steady-state bifurcation types, we compute the contractions of the second-order (Hessian, H) and third-order (T3) derivatives with respect to $u.$ 

For the contraction of the Hessian derivative with respect to $u,$ note that at steady state $x = x^*,$ we have \[\frac{\partial r_i}{\partial x_j}{\bigg |}_{x=x^*} = \frac{F_{ij}}{x_j^*}\cdot \underbrace{k_i\prod_{\ell=1}^M (x_\ell^*)^{F_{i\ell}}}_{r_i(x^*)}.\] Hence \[r_i(x^*) = \frac{\partial r_i}{\partial x_j}{\bigg |}_{x=x^*}\cdot \frac{x_j^*}{F_{ij}}.\] We rewrite our rate function in this manner so that we can obtain this information from the upper left hand submatrix of $\AA.$ Now, we have for the Hessian, if $j \ne k,$ then \[\frac{\partial^2 r_i}{\partial x_k\partial x_j}{\bigg |}_{x=x^*} = \frac{\partial}{\partial x_k}\left(\frac{F_{ij}}{x_j^*}\cdot k_i\prod_{\ell=1}^Mx_\ell^{F_{i\ell}}\right){\bigg |}_{x=x^*} = \frac{F_{ij}F_{ik}}{x_j^*x_k^*}\cdot\underbrace{k_i\prod_{\ell=1}^M (x_\ell^*)^{F_{i\ell}}}_{r_i(x^*)}.\] In contrast, if $j = k,$ then we have \[\frac{\partial^2 r_i}{\partial x_j^2} = \frac{F_{ij}(F_{ij} - 1)}{(x_j^*)^2}\cdot\underbrace{k_i\prod_{\ell=1}^m (x_i^*)^{F_{i\ell}}}_{r_i(x^*)}.\] In total, the Hessian contraction along $u$ is given by 
\begin{align*}W^\top\nabla^2_x r(x^*;\mathbf{\overline k})(u,u) &=W^\top\left(\sum_{j, k}\frac{\partial^2 r_i(x^*;\mathbf{\overline k})}{\partial x_k\partial x_j} u_j u_k\right)\\
&= W^\top\left(r_i(x^*)\left[\sum_{j\ne k}F_{ij}F_{ik}\tilde u_j\tilde u_k + \sum_{j=k}F_{ij}(F_{ij} - 1)\tilde u_j^2\right]\right)
\end{align*}where $\tilde u_j = u_j/x_j^*.$ We can simplify the summation since
\begin{align*}
    \sum_{j\ne k}F_{ij}F_{ik}\tilde u_j\tilde u_k + \sum_{j=k}F_{ij}(F_{ij} - 1)\tilde u_j^2 &=\left(\sum_{j\neq k}F_{ij}F_{ik}\tilde u_j\tilde u_k + \sum_{j=k}F_{ij}^2\tilde u_j^2-\sum_j F_{ij}\tilde u_j^2\right)\\
    &=\sum_j F_{ij}F_{ik}\tilde u_j\tilde u_k - \sum_j F_{ij}\tilde u_j^2\\
    &= (F_i^\top\cdot \tilde u)^2 - \sum_j F_{ij}\tilde u_j^2.
\end{align*} Therefore, the value corresponding to (H) may be evaluated cleanly as \[W^\top\nabla^2_x r(x^*;\mathbf{\overline k})(u,u) = W^\top\left[r_i(x^*)\left((F_i^\top\cdot \tilde u)^2 - \sum_jF_{ij}\tilde u_j^2\right)\right].\] 

Finally, for the contraction of the third-order derivative with respect to $u,$ by the form of the rate function in systems with power law kinetics, we have the following third order partial derivatives:
\begin{enumerate}
    \item if all indices are distinct so that $j\neq k\neq l,$ then \[\frac{\partial^3 r_i}{\partial x_j\partial x_k\partial x_l}{\bigg |}_{x=x^*} = \frac{F_{ij}F_{ik}F_{il}}{x^*_jx^*_kx^*_l}\cdot \underbrace{k_i\prod_{\ell = 1}^M(x_\ell^*)^{F_{i\ell}}}_{r_i(x^*)}.\]
    \item if exactly two are equal so that without loss of generality we can have $j=k\ne l,$ then we have
    \[\frac{\partial^3 r_i}{\partial x_j^2\partial x_l}{\bigg |}_{x=x^*} = \frac{F_{ij}(F_{ij} - 1)F_{il}}{(x_j^*)^2x^*_l}\cdot r_i(x^*),\] 
    \item if all indices are equal so that $j = k = l,$ then we have \[\frac{\partial^3 r_i}{\partial x_j^3}{\bigg |}_{x=x^*} = \frac{F_{ij}(F_{ij}-1)(F_{ij}-2)}{(x_j^*)^3}\cdot r_i(x^*).\]
\end{enumerate} When we sum up all these partial derivatives across all ordered triples counting multiplicities and contract along $u,$ we get \[\nabla_x^3r_i(x^*;\mathbf{\overline{k}})(u,u,u) = r_i(x^*)\left(\begin{aligned}\sum_{i,j,k~\text{distinct}}F_{ij}F_{ik}F_{il}&\tilde u_j\tilde u_k\tilde u_l + 3\sum_{j\ne l}F_{ij}(F_{ij} - 1)F_{il}\tilde u_j^2\tilde u_l\\
&+ \sum_{j}F_{ij}(F_{ij}-1)(F_{ij}-2)\tilde u_j^3\end{aligned}\right).\]

Now, define the following quantities \[A_i=F_i^\top\cdot \tilde u = \sum_jF_{ij}\tilde u_j,\quad B_i = \sum_jF_{ij}\tilde u_j^2,\text{ and}\quad C_i = \sum_jF_{ij}\tilde u_j^3,\] so we only need to show that \[\nabla_x^3r_i(x^*;\mathbf{\overline{k}})(u,u,u) = r_i(x^*)[A_i^3-3A_iB_i+2C_i].\] We have \[A_i^3 = \left(\sum_j F_{ij}\tilde u_j\right)^3 = \sum_{i,j,k~\text{distinct}}F_{ij}F_{ik}F_{il}\tilde u_j\tilde u_k\tilde u_l + 3\sum_{j\ne l}F_{ij}^2F_{il}\tilde u_j^2\tilde u_l+ \sum_jF_{ij}^3\tilde u_j^3.\] Thus, we have \[\sum_{i,j,k~\text{distinct}}F_{ij}F_{ik}F_{il}\tilde u_j\tilde u_k\tilde u_l = A_i^3 - 3\sum_{j\ne l}F_{ij}^2F_{il}\tilde u_j^2\tilde u_l - \sum_jF_{ij}^3\tilde u_j^3.\] Notice that \[\sum_{j\ne l}F_{ij}(F_{ij}-1)F_{il}\tilde u_j^2\tilde u_l =\sum_{j\ne l}F_{ij}^2F_{il}\tilde u_j^2\tilde u_l - \sum_{j\ne l}F_{ij}F_{il}\tilde u_j^2\tilde u_l,\] and \[\sum_{j}F_{ij}(F_{ij}-1)(F_{ij}-2)\tilde u_j^3 = \sum_j(F_{ij}^3-3F_{ij}^2+ 2F_{ij})\tilde u_j^3 = \sum_jF_{ij}^3\tilde u_j^3 - 3\sum_jF_{ij}^2\tilde u_j^3 + 2\sum_jF_{ij}\tilde u_j^3.\] In total, we have 
\[\nabla_x^3r_i(u,u,u) = r_i(x^*)\left(\begin{aligned}
    A_i^3 &- \cancel{3\sum_{j\ne l}F_{ij}^2F_{il}\tilde u_j^2\tilde u_l} - \cancel{\sum_jF_{ij}^3\tilde u_j^3} \\&+3\left[\cancel{\sum_{j\ne l}F_{ij}^2F_{il}\tilde u_j^2\tilde u_l} - \sum_{j\ne l}F_{ij}F_{il}\tilde u_j^2\tilde u_l\right]\\
    &\phantom{-\cancel{3\sum_{j\ne l}F_{ij}^2F_{il}\tilde u_j^2\tilde u_l}}+\cancel{\sum_jF_{ij}^3\tilde u_j^3} - 3\sum_jF_{ij}^2\tilde u_j^3 + 2\underbrace{\sum_j F_{ij}\tilde u_j^3}_{C_i}
\end{aligned}\right)\] 
which simplifies to \[\nabla_x^3r_i(u,u,u) = r_i(x^*)\left(A_i^3-3\left[\sum_{j\ne l}F_{ij}F_{il}\tilde u_j^2\tilde u_l+\sum_jF_{ij}^2\tilde u_j^3\right] + 2C_i\right).\] The quantity inside the brackets simplify to \[\begin{aligned}\sum_{j\ne l}F_{ij}F_{il}\tilde u_j^2\tilde u_l+\sum_jF_{ij}^2\tilde u_j^3 &= \sum_{j,l}F_{ij}F_{il}\tilde u_j^2\tilde u_l\\
&= \left(\sum_l F_{il}\tilde u_l\right)\left(\sum_jF_{ij}\tilde u_j^2\right) = A_iB_i,\end{aligned}\] therefore giving our desired result. Similar to the Hessian contraction, we use this expression for the implementation of the third derivative contraction with respect to $u$ along the components of $W.$

\end{proof}

\section{Reaction Network Representations of Some Biochemical Systems}\label{appendix:rn-rep}
We present here the reaction network representations of the glycolysis pathway and the Lee-Feinberg and Maclean models of the Wnt signaling pathways.
\subsection{Glycolysis Pathway}
\begin{equation}
\label{eq:glycolysis-system}
    \begin{cases}
    \begin{aligned}
    &R_1\colon \emptyset \to {\rm GLC}
    &&R_9\colon {\rm DHAP} \to {\rm GAP}\\
    &R_2\colon {\rm GLC} \to {\rm G6P}
    &&R_{10}\colon {\rm GAP} \to {\rm G13BPG}\\
    &R_3\colon {\rm G6P} \to {\rm F6P}
    &&R_{11}\colon {\rm G13BPG} \to {\rm 3PG}\\
    &R_4\colon {\rm F6P} \to {\rm G6P}
    &&R_{12}\colon {\rm 3PG} \to {\rm 2PG}\\
    &R_5\colon {\rm F6P} \to {\rm F16BP}
    &&R_{13}\colon {\rm 2PG} \to {\rm PEP}\\
    &R_6\colon {\rm F6P} \to {\rm F26BP}
    &&R_{14}\colon {\rm PEP} \to {\rm PYR}\\
    &R_7\colon {\rm F26BP} \to {\rm F6P}
    &&R_{15}\colon {\rm G6P} \to \emptyset\\
    &R_8\colon {\rm F16BP} \to {\rm DHAP} + {\rm GAP}
    &&R_{16}\colon {\rm PYR} \to \emptyset
    \end{aligned}
    \end{cases}
\end{equation}

\subsection{Wnt Signaling Pathway: Lee-Feinberg Model}
\begin{equation}
\label{eq:leefeinberg-expanded-system}
    \begin{cases}
    \begin{aligned}
    &R_1\colon {\rm Ya} \to {\rm Yi}
    &&R_{12}\colon {\rm CXY} \to {\rm CXYn}\\
    &R_2\colon {\rm Yi} \to {\rm Ya}
    &&R_{13}\colon {\rm CXYn} \to {\rm Yan} + {\rm Xp}\\
    &R_3\colon {\rm Yan} \to {\rm X}
    &&R_{14}\colon {\rm Xp} \to \emptyset\\
    &R_4\colon {\rm X} \to {\rm Yan}
    &&R_{15}\colon \emptyset \to {\rm Xn}\\
    &R_5\colon {\rm Yi} + {\rm X} \to {\rm P}
    &&R_{16}\colon {\rm Xn} \to \emptyset\\
    &R_6\colon {\rm X} \to {\rm T}
    &&R_{17}\colon \emptyset \to {\rm Di}\\
    &R_7\colon {\rm T} \to {\rm X}
    &&R_{18}\colon {\rm Di} \to \emptyset\\
    &R_8\colon {\rm CXT} + {\rm Di} \to {\rm T}
    &&R_{19}\colon {\rm Xn} + {\rm Da} \to {\rm Dan}\\
    &R_9\colon {\rm T} \to {\rm CXT} + {\rm Di}
    &&R_{20}\colon {\rm Dan} \to {\rm Xn} + {\rm Da}\\
    &R_{10}\colon {\rm Yan} + {\rm Xn} \to {\rm CXY}
    &&R_{21}\colon {\rm P} \to {\rm Yi} + {\rm T}\\
    &R_{11}\colon {\rm CXY} \to {\rm Yan} + {\rm Xn}
    \end{aligned}
    \end{cases}
\end{equation}

\subsection{Wnt Signaling Pathway: Maclean Model}
\begin{equation}
\label{eq:maclean-expanded-system}
    \begin{cases}
    \begin{aligned}
    &R_1\colon \emptyset \to {\rm X}
    &&R_{17}\colon {\rm CYDn} \to {\rm Yan} + {\rm Dan}\\
    &R_2\colon {\rm X} \to {\rm Xn}
    &&R_{18}\colon {\rm CYDn} \to {\rm Dan} + {\rm Yin}\\
    &R_3\colon {\rm Xn} \to {\rm X}
    &&R_{19}\colon {\rm Yin} + {\rm Pn} \to {\rm CYPn}\\
    &R_4\colon {\rm Ya} + {\rm X} \to {\rm CYX}
    &&R_{20}\colon {\rm CYPn} \to {\rm Yin} + {\rm Pn}\\
    &R_5\colon {\rm CYX} \to {\rm Ya} + {\rm X}
    &&R_{21}\colon {\rm CYPn} \to {\rm Yan} + {\rm Pn}\\
    &R_6\colon {\rm Yan} + {\rm Xn} \to {\rm CYXn}
    &&R_{22}\colon {\rm Da} + {\rm Ya} \to {\rm CYD}\\
    &R_7\colon {\rm CYXn} \to {\rm Yan} + {\rm Xn}
    &&R_{23}\colon {\rm CYD} \to {\rm Da} + {\rm Ya}\\
    &R_8\colon {\rm T} + {\rm Xn} \to {\rm CXT}
    &&R_{24}\colon {\rm CYD} \to {\rm Da} + {\rm Yi}\\
    &R_9\colon {\rm CXT} \to {\rm T} + {\rm Xn}
    &&R_{25}\colon {\rm Yi} + {\rm P} \to {\rm CYP}\\
    &R_{10}\colon {\rm Di} \to {\rm Da}
    &&R_{26}\colon {\rm CYP} \to {\rm Yi} + {\rm P}\\
    &R_{11}\colon {\rm Da} \to {\rm Di}
    &&R_{27}\colon {\rm CYP} \to {\rm Ya} + {\rm P}\\
    &R_{12}\colon {\rm Da} \to {\rm Dan}
    &&R_{28}\colon {\rm CYX} \to {\rm Ya}\\
    &R_{13}\colon {\rm Dan} \to {\rm Da}
    &&R_{29}\colon {\rm CYXn} \to {\rm Yan}\\
    &R_{14}\colon {\rm Yi} \to {\rm Yin}
    &&R_{30}\colon {\rm X} \to \emptyset\\
    &R_{15}\colon {\rm Yin} \to {\rm Yi}
    &&R_{31}\colon {\rm Xn} \to \emptyset\\
    &R_{16}\colon {\rm Yan} + {\rm Dan} \to {\rm CYDn}
    \end{aligned}
    \end{cases}
\end{equation}

\subsection{Eukaryotic Heatshock Response: Full Model}
\begin{equation}
\label{eq:eukaryotic-hsr-full-system}
    \begin{cases}
    \begin{aligned}
    &R_1\colon 2{\rm F} \to {\rm D}
    &&R_{10}\colon {\rm P} + {\rm D} \to {\rm PF} + {\rm F}\\
    &R_2\colon {\rm D} \to 2{\rm F}
    &&R_{11}\colon {\rm P} + {\rm T} \to {\rm PF} + 2{\rm F}\\
    &R_3\colon {\rm F} + {\rm D} \to {\rm T}
    &&R_{12}\colon {\rm P} + {\rm TE} \to {\rm PF} + {\rm E} + 2{\rm F}\\
    &R_4\colon {\rm T} \to {\rm F} + {\rm D}
    &&R_{13}\colon {\rm P} \to \emptyset\\
    &R_5\colon {\rm T} + {\rm E} \to {\rm TE}
    &&R_{14}\colon {\rm R} \to {\rm M}\\
    &R_6\colon {\rm TE} \to {\rm T} + {\rm E}
    &&R_{15}\colon {\rm P} + {\rm M} \to {\rm PM}\\
    &R_7\colon {\rm TE} \to {\rm TE} + {\rm P}
    &&R_{16}\colon {\rm PM} \to {\rm P} + {\rm M}\\
    &R_8\colon {\rm P} + {\rm F} \to {\rm PF}
    &&R_{17}\colon {\rm PM} \to {\rm P} + {\rm R}\\
    &R_9\colon {\rm PF} \to {\rm P} + {\rm F}
    \end{aligned}
    \end{cases}
\end{equation}

\subsection{Eukaryotic Heatshock Response: Extended Model}
\begin{equation}
\label{eq:eukaryotic-hsr-extended-subset}
    \begin{cases}
    \text{(Reactions } R_1 \text{ through } R_{17} \text{ are identical to the full model)}\\
    R_{18}\colon {\rm P} \to {\rm Y}\\
    R_{19}\colon {\rm F} \to {\rm X}\\
    R_{20}\colon {\rm P} + {\rm Y} \to {\rm PY}\\
    R_{21}\colon {\rm PY} \to {\rm P} + {\rm Y}\\
    R_{22}\colon {\rm P} + {\rm X} \to {\rm PX}\\
    R_{23}\colon {\rm PX} \to {\rm P} + {\rm X}\\
    R_{24}\colon {\rm PY} \to 2{\rm P}\\
    R_{25}\colon {\rm PX} \to {\rm P} + {\rm F}
    \end{cases}
\end{equation}

\pagebreak
\section{Supplementary Tables}\label{appendix:tables}
\begin{table*}[ht]
    \caption{Symbolic expressions of the net parameter sensitivity term $\Theta$ when the parameters for the reactions in $\Gamma_8$ of the glycolysis pathway are varied. All are structurally nonzero, implying saddle-node bifurcations (may be degenerate or cusp).}
    \label{tab:theta_glycolysis}
    \centering
    \(\displaystyle 
    \renewcommand{\arraystretch}{1.9}\begin{array}{cc}
        \hline\text{Reaction \#}& \Theta \\\hline
        3 & G6P^* \dfrac{-r_{15,G6P}}{r_{3,G6P}}\\
        4 & F6P^* \dfrac{-r_{15,G6P}}{r_{3,G6P}}\\
        5 & F6P^* F26BP^*\dfrac{-r_{15,G6P} r_{4,F6P} r_{7, F26BP}}{r_{3,G6P}(r_{6,F6P} r_{5,F26BP} + r_{5,F6P} r_{7, F26BP}}\\
        6 & F6P^*\dfrac{-r_{15,G6P} r_{4,F6P} r_{5,F26BP}}{r_{3,G6P} PEP^*(r_{6,F6P} r_{5,F26BP} + r_{5,F6P} r_{7,F26BP})}\\
        7 & F26BP^*\dfrac{r_{15,G6P} r_{4,F6P} r_{5,F26BP}}{r_{3,G6P}(r_{6,F6P}r_{5,F26BP} + r_{5,F6P}r_{7,F26BP}}\\
        8 & F16BP^*\dfrac{(r_{14,PEP} r_{15,G6P} r_{6,F6P} r_{5,F26BP} + r_{14,PEP} r_{15,G6P} r_{4,F6P} r_{7,F26BP} + \ldots)}{r_{14,PEP} r_{3,G6P} (r_{6,F6P} r_{5,F26BP} + r_{5,F6P} r_{7,F26BP})}\\
        14 & F16BP^* PEP^*\dfrac{-r_{15,G6P} r_{4,F6P} r_{6,PEP} r_{5,F26BP}}{r_{14,PEP} r_{3,G6P}(r_{6,F6P} r_{5,F26BP} + r_{5,F6P} r_{7,F26BP})}\\
        15 & G6P^*\\\hline
    \end{array}\)
\end{table*}
\begin{table*}[ht]
    \centering
    \caption{Symbolic expressions of the transversality term $\Theta$ for all reactions in the DOC model. Under the assumption that $r_{1,{\rm Land}} - r_{2,{\rm Land}} \ne 0,$ all reactions except for $R_7\colon {\rm Oce}\to {\rm DOC}$ are structurally incapable of vanishing implying saddle-node bifurcations.}
    \( 
    \renewcommand{\arraystretch}{1.9}\begin{array}{cc}
        \hline\text{Reaction \#}& \Theta \\\hline
        1 & ({\rm Atm}^*)^2{\rm (Land}^*) \dfrac{1}{r_{1,{\rm Land}} - r_{2, {\rm Land}}}\\
        2 & -({\rm Atm}^*)^2{\rm (Land}^*) \dfrac{1}{r_{1,{\rm Land}} - r_{2, {\rm Land}}}\\
        3 & -{\rm Atm}^* \dfrac{(r_{6, {\rm DOC}}r_{5, {\rm LTS}} + r_{6, {\rm DOC}}r_{7, {\rm Oce}} + r_{5, {\rm LTS}}r_{7, {\rm Oce}})}{r_{6, {\rm DOC}}r_{5, {\rm LTS}}(r_{4, {\rm Oce}} + r_{7, {\rm Oce}})}\\
        4 & {\rm Oce}^* \dfrac{r_{6, {\rm DOC}}r_{5, {\rm LTS}} + r_{6, {\rm DOC}}r_{7, {\rm Oce}} + r_{5, {\rm LTS}}r_{7, {\rm Oce}}}{r_{6, {\rm DOC}}r_{5, {\rm LTS}}(r_{4, {\rm Oce}} + r_{7, {\rm Oce}})}\\
        5 & {\rm LTS}^*/{r_{5, {\rm LTS}}}\\
        6 & -1/{r_{6, {\rm DOC}}^2}\\
        7 & -({\rm Oce}^*) \dfrac{r_{6, {\rm DOC}}r_{4, {\rm Oce}} - r_{6, {\rm DOC}}r_{5, {\rm LTS}} + r_{5, {\rm LTS}}r_{4, {\rm Oce}}}{r_{6, {\rm DOC}}r_{5, {\rm LTS}}(r_{4, {\rm Oce}} + r_{7, {\rm Oce}})}\\
        \hline
    \end{array}\)
    
    \label{tab:theta_DOC}
\end{table*}

\bibliography{references}
\end{document}